\documentclass[11pt]{amsart}
\usepackage{latexsym,amscd,amssymb, graphicx, color, amsthm, bm, amsmath, cancel, enumitem, ytableau}  
\usepackage{tikz}
\usepackage[margin=1in]{geometry}
\usepackage{hyperref}
\usepackage[all,cmtip]{xy}
\usepackage{mathtools}
\usepackage{rotating}

\numberwithin{equation}{section}

\newtheorem{theorem}{Theorem}[section]
\newtheorem{proposition}[theorem]{Proposition}
\newtheorem{corollary}[theorem]{Corollary}
\newtheorem{lemma}[theorem]{Lemma}
\newtheorem{conjecture}[theorem]{Conjecture}

\newtheorem{problem}[theorem]{Problem}

\newtheorem{remark}[theorem]{Remark}

\theoremstyle{definition}
\newtheorem{defn}[theorem]{Definition}

\newcommand{\Hilb}{{\mathrm{Hilb}}}

\newcommand{\maj}{{\mathrm{maj}}}

\newcommand{\one}{{\mathbf{1}}}

\newcommand{\Frob}{{\mathrm{Frob}}}
\newcommand{\symm}{{\mathfrak{S}}}
\newcommand{\II}{{\mathbf{I}}}

\newcommand{\MMM}{{\mathcal{M}}}
\newcommand{\gr}{{\mathrm {gr}}}
\newcommand{\BBB}{{\mathcal{B}}}
\newcommand{\KKK}{{\mathcal{K}}}
\newcommand{\PM}{{\mathcal{PM}}}

\newcommand{\Res}{\mathrm{Res}}

\newcommand{\lds}{\mathrm{lds}}

\newcommand{\Zpoints}{\mathcal{Z}}

\newcommand{\grFrob}{{\mathrm{grFrob}}}
\newcommand{\Ind}{{\mathrm{Ind}}}

\newcommand{\CC}{{\mathbb{C}}}

\newcommand{\ZZ}{{\mathbb{Z}}}

\newcommand{\Mat}{{\mathrm{Mat}}}

\newcommand{\SYT}{\mathrm{SYT}}
\newcommand{\spa}{\mathrm{span}}

\newcommand{\xxx}{{\mathbf{x}}}
\newcommand{\zzz}{{\mathbf{z}}}

\newcommand{\mmm}{{\mathfrak{m}}}

\newcommand{\bl}{\boldsymbol{\lambda}}

\newcommand{\precdot}{\mathrel{\ooalign{$\prec$\cr
  \hidewidth\raise0.001ex\hbox{$\cdot\mkern0.6mu$}\cr}}}

\title{Involution Matrix Loci and Orbit Harmonics}
\author{Moxuan (Jasper) Liu}
\address{Department of Mathematics \\ UC San Diego \\ La Jolla, CA, 92093 \\ USA}
\email{(mol008, bprhoades)@ucsd.edu}
\author{Yichen Ma}
\address{Department of Mathematics \\ Cornell University \\ Ithaca, NY, 14853 \\ USA}
\email{ym476@cornell.edu}
\author{Brendon Rhoades}
\author{Hai Zhu}
\address{School of Mathematical Sciences \\ Peking University \\ Haidian District, Beijing 100871 \\ China}
\email{zhuhai1686@stu.pku.edu.cn}

\begin{document}

\maketitle

\begin{abstract}
    Let $\Mat_{n \times n}(\CC)$ be the affine space of $n \times n$ complex matrices with coordinate ring $\CC[\xxx_{n \times n}]$. We define graded quotients of $\CC[\xxx_{n \times n}]$ which carry an action of the symmetric group $\symm_n$ by simultaneous permutation of rows and columns. These quotient rings are obtained by applying the orbit harmonics method to matrix loci corresponding to all involutions in $\symm_n$ and the conjugacy classes of involutions in $\symm_n$ with a given number of fixed points. In the case of perfect matchings on $\{1, \dots, n\}$ with $n$ even, the Hilbert series of our quotient ring is related to Tracy-Widom distributions and its graded Frobenius image gives a refinement of the plethysm $s_{n/2}[s_2]$.
\end{abstract}

\section{Introduction}
\label{sec:Introduction}

Let $\xxx_N = (x_1, \dots, x_N)$ be a list of $N$ variables and let $\CC[\xxx_N] := \CC[x_1,\dots,x_N]$ be the polynomial ring over these variables. If $\Zpoints \subseteq \CC^N$ is a finite locus of points in affine $N$-space, we have the vanishing ideal
\begin{equation}
    \II(\Zpoints) := \{ f \in \CC[\xxx_N] \,:\, f(\zzz) = 0 \text{ for all $\zzz \in \Zpoints$} \}
\end{equation}
and an identification $\CC[\Zpoints] = \CC[\xxx_N]/\II(\Zpoints)$ of vector spaces. The method of {\em orbit harmonics} replaces $\II(\Zpoints)$ by its associated graded ideal $\gr \, \II(\Zpoints)$; we have an isomorphism
\begin{equation}
\label{eqn:orbit-harmonics-isomorphism}
    \CC[\Zpoints] = \CC[\xxx_N]/\II(\Zpoints) \cong \CC[\xxx_N]/\gr \, \II(\Zpoints) = R(\Zpoints)
\end{equation}
where the quotient $R(\Zpoints) := \CC[\xxx_N]/\gr \, \II(\Zpoints)$ is a graded $\CC$-vector space. If $\Zpoints$ is stable under the action of a finite matrix group $G \subseteq GL_N(\CC)$, then \eqref{eqn:orbit-harmonics-isomorphism} is an isomorphism of $G$-modules, where $R(\Zpoints)$ is a graded $G$-module.

Geometrically, orbit harmonics corresponds to linearly deforming the reduced locus $\Zpoints$ to a scheme of multiplicity $|\Zpoints|$ supported at the origin. This deformation is shown schematically in the picture below for a locus $\Zpoints$ of six points in the plane which is stable under the group $(\cong \symm_3)$ generated by reflections in the lines. 

\begin{center}
 \begin{tikzpicture}[scale = 0.2]
\draw (-4,0) -- (4,0);
\draw (-2,-3.46) -- (2,3.46);
\draw (-2,3.46) -- (2,-3.46);

 \fontsize{5pt}{5pt} \selectfont
\node at (0,2) {$\bullet$};
\node at (0,-2) {$\bullet$};

\node at (-1.73,1) {$\bullet$};
\node at (-1.73,-1) {$\bullet$};
\node at (1.73,-1) {$\bullet$};
\node at (1.73,1) {$\bullet$};

\draw[thick, ->] (6,0) -- (8,0);

\draw (10,0) -- (18,0);
\draw (12,-3.46) -- (16,3.46);
\draw (12,3.46) -- (16,-3.46);

\draw (14,0) circle (15pt);
\draw(14,0) circle (25pt);
\node at (14,0) {$\bullet$};

 \end{tikzpicture}
\end{center}

Orbit harmonics  has seen applications to presenting cohomology rings \cite{GP}, Macdonald theory \cite{Griffin, HRS}, cyclic sieving \cite{OR}, Donaldson-Thomas theory \cite{RRT}, and Ehrhart theory \cite{RR}. When the locus $\Zpoints$ has favorable `organization' and interesting symmetries, the quotient $R(\Zpoints)$ often has nice properties. One expects algebraic properties of $R(\Zpoints)$ to be governed by combinatorial properties of $\Zpoints$.

Let $\xxx_{n \times n} = (x_{i,j})_{1 \leq i,j \leq n}$ be an $n \times n$ matrix of variables. Consider the affine space $\Mat_{n \times n}(\CC)$ of $n \times n$ complex matrices with coordinate ring $\CC[\xxx_{n \times n}]$. The application of orbit harmonics to finite matrix loci $\Zpoints \subseteq \Mat_{n \times n}(\CC)$ was initiated by Rhoades \cite{RhoadesViennot}. He considered the locus $\Zpoints = \symm_n \subseteq \Mat_{n \times n}(\CC)$ of $n \times n$ permutation matrices. Since $\symm_n$ forms a group, this permutation matrix locus carries an action of $\symm_n \times \symm_n$ by left and right multiplication. Algebraic properties of $R(\symm_n)$ are governed by longest increasing subsequences in $\symm_n$ and the Schensted correspondence. Liu extended \cite{Liu} this work to the locus $\Zpoints = \ZZ_r \wr \symm_n$ of $r$-colored permutation matrices. 

In this paper we consider matrix loci $\Zpoints \subseteq \Mat_{n \times n}(\CC)$ which do not form a group, but are closed under permutation matrix conjugation. We consider the  locus 
\begin{equation}
    \MMM_n := \{ w \in \symm_n \,:\, w^2 = 1 \} \subseteq \Mat_{n \times n}(\CC)
\end{equation}
of involutions (or matchings). For $0 \leq a \leq n$ with $a \equiv n \mod 2$, we have the sublocus
\begin{equation}
    \MMM_{n,a} := \{ w \in \MMM_n \,:\, \text{$w$ has exactly $a$ fixed points} \}
\end{equation}
of matchings with $a$ fixed points. If $n$ is even, we write
\begin{equation}
\PM_n := \MMM_{n,0}  \quad \quad \text{($n$ even)}
\end{equation}
for the set of perfect matchings (or fixed-point-free involutions).
The loci $\MMM_n$ and $\MMM_{n,a}$ are closed under the conjugation action of the permutation matrix group $\symm_n$. The quotients $R(\MMM_n)$ and $R(\MMM_{n,a})$ are therefore graded $\symm_n$-modules. Our results on these modules are as follows.
\begin{itemize}
    \item We give explicit generating sets for the associated graded ideals $\gr \, \II(\MMM_n)$ and $\gr \, \II(\PM_n)$; see Propositions~\ref{prop:matching-ideal-equality}, \ref{prop:pm-ideal-equality}.
    \item We show (Theorem~\ref{thm:matching-monomial-basis}) that $R(\MMM_n)$ admits a basis of monomials $\mmm(w)$ indexed by matchings $w \in \MMM_n$ where 
    $$\mmm(w) := \prod_{\substack{w(i) = j \\ i \, < \, j}} x_{i,j}$$
    is the squarefree product of upper triangular variables indexed by 2-cycles in $w$.
    \item We describe the graded $\symm_n$-module structure of each of the orbit harmonics quotient rings $R(\MMM_n)$ and $R(\MMM_{n,a})$; see Theorems~\ref{thm:matching-frobenius}, \ref{thm:conjugacy-module-character}.
\end{itemize}
In the case of $R(\PM_n)$, the graded structure in the last bullet point has an especially nice form. For $n$ even, the conjugation action of $\symm_n$ on perfect matchings is well-known to have Frobenius image 
\begin{equation}
    \Frob(\PM_n) = s_{n/2}[s_2] = \sum_{\substack{\lambda \, \vdash \, n \\ \lambda \text{ even}}} s_\lambda
\end{equation}
where $s_{n/2}[s_2]$ is the plethysm of Schur functions and the sum is over partitions $\lambda \vdash n$ with only even parts. The graded refinement of this character weights this sum according to the length of the first row of $\lambda$, viz.
\begin{equation}
    \grFrob(R(\PM_n);q) = \sum_{\substack{\lambda \, \vdash \, n \\ \lambda \text{ even}}} q^{\frac{n - \lambda_1}{2}} \cdot s_\lambda
\end{equation}
where $q$ is a grading variable; see Theorem~\ref{thm:pm-frobenius}. The Hilbert series of $R(\PM_n)$ is (up to renormalization) the distribution of the longest decreasing subsequence statistic on fixed-point-free involutions in $\symm_n$; see Corollary~\ref{cor:pm-hilbert}. Baik and Rains proved \cite{BR} that the limit as $n \to \infty$ of this Hilbert series arises in random matrix theory as a flavor of the Tracy-Widom distribution (up to renormalization and reversal).

The rest of the paper is organized as follows. 
In {\bf Section~\ref{sec:Background}} we give background material on the orbit harmonics deformation, symmetric functions, and $\symm_n$-representation theory.
In {\bf Section~\ref{sec:Matching}} we prove our results on the matrix locus $\MMM_n$ of all involutions in $\symm_n$.
In {\bf Section~\ref{sec:Perfect}} we turn to the perfect matching locus $\PM_n$. {\bf Section~\ref{sec:Conjugacy}} is the most technical part of the paper; it computes the graded $\symm_n$-structure of $R(\MMM_{n,a})$. We close in {\bf Section~\ref{sec:Conclusion}} with some conjectures and open problems.

\section{Background}
\label{sec:Background}


\subsection{Orbit Harmonics}
Let $\Zpoints \subseteq \CC^N$ be a finite locus of points in affine $N$-space. As in the introduction, the  {\em vanishing ideal} $\II(\Zpoints)$ is given by
\begin{equation}
    \II(\Zpoints) := \{f \in \CC[\xxx_N] : f(\zzz) = 0 \text{ for all  $\zzz \in \Zpoints$} \}.
\end{equation}
Multivariate Lagrange interpolation gives the identification of vector spaces
\begin{equation}\label{eq:locus-vanishing quotient}
    \CC[\Zpoints] \cong \CC[\xxx_N]/\II(\Zpoints)
\end{equation}
where $\CC[\Zpoints]$ is the space of all functions $\Zpoints \to \CC$. If $\Zpoints$ is stable under the action of a finite matrix group $G \subseteq GL_N(\CC)$, the ideal $\II(\Zpoints)$ is also $G$-stable and \eqref{eq:locus-vanishing quotient} is an isomorphism of $G$-modules.

Given a nonzero polynomial $f \in \CC[\xxx_N]$, write $f = f_0 + f_1 + \dots + f_d$, where $f_i$ is homogeneous of degree $i$ and $f_d \neq 0$. Let $\tau(f)$ be the top degree homogeneous part of $f$, that is, $\tau(f) := f_d$. Given an ideal $I \subseteq \CC[\xxx_N]$, the {\em associated graded ideal} of $I$ is the homogeneous ideal
\begin{equation}
    \gr \, I := ( \tau(f) : f \in I, f \neq 0 ).
\end{equation}

\begin{remark}
    \label{rmk:gr-generation}
    Given a generating set $I = (f_1, \dots, f_r)$ of an ideal $I \subseteq \CC[\xxx_N]$, we have a containment $(\tau(f_1), \dots, \tau(f_r)) \subseteq \gr \, I$ of homogeneous ideals in $\CC[\xxx_N]$, but this containment is strict in general. 
\end{remark}

The identification of ungraded $\CC$-vector spaces in Equation~\eqref{eq:locus-vanishing quotient} extended to a vector space isomorphism
\begin{equation}\label{eq:orbit-harmonics-isomorphisms}
    \CC[\Zpoints] \cong \CC[\xxx_N]/\II(\Zpoints) \cong \CC[\xxx_N]/\gr \, \II(\Zpoints) =: R(\Zpoints).
\end{equation}
Since $\gr \, \II(\Zpoints) \subseteq \CC[\xxx_N]$ is a homogeneous ideal, $R(\Zpoints)$ has the additional structure of a graded $\CC$-vector space. If $\Zpoints$ is stable under the action of a finite matrix group $G \subseteq GL_n(\CC)$, we may regard \eqref{eq:orbit-harmonics-isomorphisms} as an isomorphism of ungraded $G$-modules, where $R(\Zpoints)$ has the additional structure of a graded $G$-module.

We will need a slightly more refined statement than the isomporphism \eqref{eq:orbit-harmonics-isomorphisms}. If $V  = \bigoplus_{d \geq 0} V_d$ is a graded vector space and $m \geq 0$, we write $V_{\leq m} := \bigoplus_{d = 0}^m V_d$. In particular, we write $\CC[\xxx_N]_{\leq m}$ for the vector space of polynomials in $\CC[\xxx_N]$ of degree $\leq m$. If $I \subseteq \CC[\xxx_N]$ is an ideal, we have a subspace $I \cap \CC[\xxx_{N}]_{\leq m} \subseteq \CC[\xxx_N]_{\leq m}$. The following result is standard; see e.g. \cite[Lem. 3.15]{RhoadesViennot}.

\begin{lemma}
    \label{lem:basis-transfer}
    Suppose $I \subseteq \CC[\xxx_N]$ is an ideal and $m \geq 0$. Let $\BBB \subseteq \CC[\xxx_N]_{\leq m}$ be a set of homogeneous polynomials which descends to a vector space basis for $(\CC[\xxx_N]/\gr \, I)_{\leq m}$. Then $\BBB$ descends to a vector space basis of $\CC[\xxx_N]_{\leq m}/(I \cap \CC[\xxx_N]_{\leq m})$.
\end{lemma}

Now suppose $\Zpoints \subseteq \CC^N$ is a finite locus which is stable under the action of a finite matrix group $G \subseteq GL_N(\CC)$. For $m \geq 0$, the subspace $\CC[\xxx_N]_{\leq m} \subseteq \CC[\xxx_N]$ is $G$-stable, as is the ideal $\II(\Zpoints)$. The quotient $\CC[\xxx_N]_{\leq m}/(I \cap \CC[\xxx_N]_{\leq m})$ is therefore an ungraded $G$-module. The following result refines the isomorphism \eqref{eq:orbit-harmonics-isomorphisms}.

\begin{lemma}
    \label{lem:refined-orbit-harmonics}
    Suppose a finite locus $\Zpoints \subseteq \CC^N$ is stable under the action of a finite matrix group $G \subseteq GL_N(\CC)$. For $m \geq 0$, we have an isomorphism of ungraded $G$-modules
    $$R(\Zpoints)_{\leq m} = (\CC[\xxx_N]/\gr \, \II(\Zpoints))_{\leq m} \cong_G \CC[\xxx_N]_{\leq m}/(I \cap \CC[\xxx_N]_{\leq m}).$$
\end{lemma}

Lemma~\ref{lem:refined-orbit-harmonics} is well-known in orbit harmonics theory, but we include a proof for the benefit of the reader.

\begin{proof}
    Let $\BBB \subseteq \CC[\xxx_N]_{\leq m}$ be a set of homogeneous polynomials which descends to a vector space basis of $R(\Zpoints)_{\leq m}$. Lemma~\ref{lem:basis-transfer} says that $\BBB$ also descends to a vector space basis of the quotient $\CC[\xxx_N]_{\leq m}/(\II(\Zpoints) \cap \CC[\xxx_N]_{\leq m}).$ 
    
    Fix a matrix $g \in G$.  Write $\BBB = \BBB_0 \sqcup \BBB_1 \sqcup \cdots \sqcup \BBB_m$ where $\BBB_d$ denotes the elements in $\BBB$ of homogeneous degree $d$. One checks that 
    \begin{itemize}
        \item the representing matrix for $g$ on $R(\Zpoints)_{\leq m}$ is block diagonal with respect to the stratification $\BBB = \bigsqcup_{d = 0}^m \BBB_d$, 
        \item the representing matrix for $g$ on $\CC[\xxx_N]_{\leq m}/(\II(\Zpoints) \cap \CC[\xxx_N]_{\leq m})$ is block triangular with respect to the stratification $\BBB = \bigsqcup_{d = 0}^m \BBB_d$, and
        \item the diagonal blocks of the above two matrices coincide.
    \end{itemize}
    By the above three bullet points, the $\CC[G]$-modules $R(\Zpoints)_{\leq m}$ and $\CC[\xxx_N]_{\leq m}/(\II(\Zpoints) \cap \CC[\xxx_N]_{\leq m})$ have the same composition series. Since $G$ is a finite group, the group algebra $\CC[G]$ is semisimple and we have $R(\Zpoints)_{\leq m} \cong_G \CC[\xxx_N]_{\leq m}/(\II(\Zpoints) \cap \CC[\xxx_N]_{\leq m})$, as desired.
\end{proof}

The above proof works over any field $\mathbb{F}$ in which the order of $G$ is nonzero. For arbitrary fields $\mathbb{F}$, the $\mathbb{F}[G]$-modules $(\mathbb{F}[\xxx_N]/\gr \, \II(\Zpoints))_{\leq m}$ and $\mathbb{F}[\xxx_N]_{\leq m}/(\II(\Zpoints) \cap \mathbb{F}[\xxx_N]_{\leq m})$ have the same composition series (i.e. these modules are {\em Brauer-isomorphic}).

\subsection{The Schensted Correspondence}
Let $n$ be a positive integer. A {\em partition} of $n$ is a sequence of non-decreasing positive integers $\lambda = (\lambda_1, \lambda_2, \dots , \lambda_k)$ such that $\sum_{i=1}^k \lambda_i = n$. We write $\lambda \vdash n$ to say that $\lambda$ is a partition of $n$.

Let $\lambda = (\lambda_1, \lambda_2, \dots , \lambda_k) \vdash n$. A {\em Young diagram} of shape $\lambda$ is the diagram obtained by placing $\lambda_i$ boxes in the $i$-th row. Below on the left is a Young diagram of shape $(4,2,1) \vdash 7$.
\begin{center}
    \begin{ytableau}
    \: & \: & \: & \: \cr
       & \cr
    \:
\end{ytableau} \quad \quad
\begin{ytableau}
    1 & 3 & 6 & 7 \cr
    2 & 5 \cr
    4
\end{ytableau}
\end{center}
A {\em standard Young tableaux} of shape $\lambda$ is a bijective filling of $[n]$ into the boxes of the Young diagram of shape $\lambda$, such that the entries are increasing across rows and down columns. Above on the right is a standard Young tableaux of shape $(4,2,1)$. We write $\SYT(\lambda)$ to denote the collection of standard Young tableau of shape $\lambda$.

The {\em Schensted correspondence}~\cite{Schensted} is a bijection
\begin{equation}
    \symm_n \xrightarrow{\quad \sim \quad} \bigsqcup_{\lambda \vdash n} \{ (P,Q): P,Q \in \SYT (\lambda)\},
\end{equation}
which associates each permutation $w \in \symm_n$ with a pair of standard Young tableau $(P(w),Q(w))$ of the same shape. The Schensted correspondence is usually described by an insertion algorithm, which can be found in \cite{Sagan}. 

\begin{remark}
    One important fact about the Schensted correspondence is that if $w \in \symm_n$ corresponds to the pair of tableau $(P,Q)$, then $w^{-1}$ corresponds to the pair of tableau $(Q,P)$. In particular, if $w \in \symm_n$ is an involution, the pair of tableau that corrsponds to $w$ will have the form $(P,P)$.
\end{remark}

\subsection{Representation Theory of $\symm_n$} We write $\Lambda = \bigoplus_{n \, \geq \, 0} \Lambda_n$ for the graded algebra of symmetric functions in an infinite variable set $\xxx = (x_1, x_2, \dots )$ over the ground field $\CC(q)$. For example, the {\em power sum symmetric function} of degree $n$ is given by $p_n := \sum_{i \geq 1} x_i^n \in \Lambda_n$. It is well known that $\{p_1, p_2, \dots \}$ is an algebraically independent generating set $\Lambda$.

Bases of the $n^{th}$ graded piece $\Lambda_n$ of $\Lambda$ are indexed by partitions $\lambda \vdash n$. For example, the {\em power sum basis} $\{ p_\lambda \,:\, \lambda \vdash n \}$ is defined by $p_\lambda := p_{\lambda_1} p_{\lambda_2} \cdots$ for $\lambda = (\lambda_1 \geq \lambda_2 \geq \cdots )$. We will also need the {\em Schur basis} $\{ s_\lambda \,:\, \lambda \vdash n\}$.

A symmetric function $F \in \Lambda$ is {\em Schur-positive} if the coefficients $c_\lambda(q)$ in its $s$-expansion $F = \sum_\lambda c_\lambda(q) \cdot s_\lambda$ satisfy $c_\lambda(q) \in \mathbb{R}_{\geq 0}[q]$. If $F, G \in \Lambda$ we define $F \leq G$ by
\begin{center}
    $F \leq G$ if and only if $G-F$ is Schur-positive.
\end{center}
Furthermore, if $P$ is some property that partitions may or may not have and 
$F = \sum_\lambda c_\lambda(q) \cdot s_\lambda$, we define the truncation $\{F\}_P$ by
\begin{equation}
    \{F\}_P := \sum_{\lambda} \chi[\text{$\lambda$ satisfies $P$}] \cdot c_\lambda(q) \cdot s_\lambda.
\end{equation}
Here $\chi[S]$ is 1 if a statement $S$ is true and 0 otherwise. The truncation operator $\{-\}_P$ will typically be used to restrict the first row lengths of a partition, e.g. $\{ - \}_{\lambda_1 = a}$ or $\{ - \}_{\lambda_1 \leq a}$. 

In addition to its usual multiplication, the graded ring $\Lambda$ carries an additional binary operation called {\em plethysm}. If $F = F(x_1, x_2, \dots ) \in \Lambda$ and $n > 0$, we defined $p_n[F] \in \Lambda$ by 
\begin{equation}
    p_n[F] := F(x_1^n, x_2^n, \dots ).
\end{equation}
Since $\{p_1, p_2, \dots \}$ is an algebraically independent generating set of $\Lambda$, the rules
\begin{equation}
    \begin{cases}
        (G_1 + G_2)[F] = G_1[F] + G_2[F] \\
        (G_1 \cdot G_2)[F] = G_1[F] \cdot G_2[F] \\
        c[F] = c
    \end{cases}
\end{equation}
for all $G_1, G_2, F \in \Lambda$ and all scalars $c \in \CC(q)$ uniquely define a symmetric function $G[F] \in \Lambda$ for all $G, F \in \Lambda$. For more details on plethysm, see Macdonald's book \cite{Macdonald}. 

Irreducible representations of the symmetric group $\symm_n$ over $\CC$ are in one-to-one correspondence with partitions $\lambda \vdash n$. If $\lambda \vdash n$ is a partition, write $V^\lambda$ for the corresponding irreducible $\symm_n$-module. If $V$ is any finite-dimensional $\symm_n$-module, there are unique multiplicities $c_\lambda \geq 0$ so that $V \cong \bigoplus_{\lambda \vdash n} c_\lambda V^\lambda$. The {\em Frobenius image} of $V$ is the symmetric function
\begin{equation}
    \Frob(V) := \sum_{\lambda \, \vdash \, n} c_\lambda \cdot s_\lambda
\end{equation}
obtained by replacing each irreducible $V^\lambda$ with the corresponding Schur function $s_\lambda$. For example, the trivial representation $\one_{\symm_n}$ corresponds to the Schur function $s_{n}$ indexed by the one-part partition $(n)$. If $V = \bigoplus_{d \geq 0} V_d$ is a graded $\symm_n$-module, we define the {\em graded Frobenius image} by 
\begin{equation}
    \grFrob(V;q) := \sum_{d \, \geq \, 0} \Frob(V_d) \cdot q^d.
\end{equation}
More generally, if $V = \bigoplus_{d \geq 0} V_d$ is a graded vector space, the {\em Hilbert series} of $V$ is
\begin{equation}
    \Hilb(V;q) := \sum_{d \, \geq \, 0} \dim(V_d) \cdot q^d.
\end{equation}

Given $n, m \geq 0$, we have an embedding $\symm_n \times \symm_m \subseteq \symm_{n+m}$ by letting $\symm_n$ permute the first $n$ letters of $[n+m]$ and letting $\symm_m$ permute the last $m$ letters. If $V$ is an $\symm_n$-module an $W$ is an $\symm_m$-module, the tensor product $V \otimes W$ is naturally an $\symm_n \times \symm_m$-module. The {\em induction product} of $V$ and $W$ is 
\begin{equation}
    V \circ W := \Ind_{\symm_n \times \symm_m}^{\symm_{n+m}} (V \otimes W).
\end{equation}
Induction product of representations corresponds to multiplication of symmetric functions; we have (see e.g. \cite{Macdonald, Sagan})
\begin{equation}
    \Frob(V \circ W) = \Frob(V) \cdot \Frob(W).
\end{equation}

For $n, m \geq 0$, the {\em wreath product} $\symm_n \wr \symm_m$ is the subgroup of $\symm_{nm}$ generated by $\dots$
\begin{itemize}
    \item the $n$-fold direct product $\symm_m \times \symm_m \times \cdots \times \symm_m$ permuting elements of the sets $$\{1,2,\dots,m\},\{m+1, m + 2, \dots,2m\}, \dots, \{n m - m +1, nm - m + 2, \dots, n m\}$$
    independently, as well as
    \item products of transpositions of the form $$(im-m+1,jm-m+1)(im-m+2,jm-m+2) \cdots (im,jm)$$
    for $1 \leq i,j \leq n$ which interchange two of the above $n$ sets `wholesale'.
\end{itemize}
One way to visualize the group $\symm_n \wr \symm_m$ is as follows. Fill the entries of an $n \times m$ grid with the numbers $1, 2, \dots, nm$ in English reading order. For example, if $n = 3$ and $m = 4$ we have the following figure.
$$\begin{ytableau}
    1 & 2 & 3 & 4 \\
    5 & 6 & 7 & 8 \\
    9 & 10 & 11 & 12
\end{ytableau}$$
The group $\symm_n \wr \symm_m$ is the subgroup of $\symm_{nm}$ generated by permutations which either permute entries within rows or interchange rows wholesale.
We have embeddings $\symm_m^n \subseteq \symm_n \wr \symm_m$ and $\symm_n \subseteq \symm_n \wr \symm_m$; the images of these embeddings generate $\symm_n \wr \symm_m$.

If $V$ is an $\symm_n$-module and $W$ is an $\symm_m$-module, we have a $(\symm_n \wr \symm_m)$-module $V \wr W$ with underlying vector space 
\begin{equation}
    V \wr W := V \otimes \overbrace{W \otimes \cdots \otimes W}^n
\end{equation}
and group action determined by
\begin{equation}    
    \begin{cases}
        (u_1, \dots, u_n) \cdot (v \otimes w_1 \otimes \cdots \otimes w_n) := v \otimes (u_1 \cdot w_1) \otimes \cdots \otimes (u_n \cdot w_n) & (u_1, \dots, u_n) \in \symm_m^n, \\
        t \cdot (v \otimes w_1 \otimes \cdots \otimes w_n) := t \cdot v \otimes w_{t^{-1}(1)} \otimes \cdots \otimes w_{t^{-1}(n)} & t \in \symm_n.
    \end{cases}
\end{equation}
Plethysm of symmetric functions relates to this construction via
\begin{equation}
    \Frob(V)[\Frob(W)] = \Ind_{\symm_n \wr \symm_m}^{\symm_{nm}} (V \wr W).
\end{equation}
See \cite{Macdonald} for a textbook treatment of this fact.
We will only use this construction when both $V$ and $W$ are trivial modules; we note the combinatorial interpretation of the relevant symmetric function $s_n[s_m]$.

\begin{lemma}
    \label{lem:plethysm-interpretation}
    For $n, m \geq 0$, let $\Pi_{n,m}$ be the family of set partitions of $[nm]$ consisting of $n$ blocks, each of size $m$. The symmetric group $\symm_{nm}$ acts on $\Pi_{n,m}$; let $\CC[\Pi_{n,m}]$ be the corresponding permutation module. The $\symm_{nm}$-module $\CC[\Pi_{n,m}]$ has Frobenius image
    $$\Frob(\Pi_{n,m}) = s_n[s_m].$$
\end{lemma}

\begin{proof}
    It is not hard to see that $\one_{\symm_n} \wr \one_{\symm_m} = \one_{\symm_n \wr \symm_m}$ so that $s_n[s_m]$ is the Frobenius image of the coset representation
    \begin{equation}
        \Ind_{\symm_n \wr \symm_m}^{\symm_{nm}} \one_{\symm_n \wr \symm_m} = \CC[\symm_{nm}/(\symm_n \wr \symm_m)].
    \end{equation}
    On the other hand, there is a natural correspondence between set partitions in $\Pi_{n,m}$ and cosets in $\symm_{nm}/(\symm_n \wr \symm_m)$ which is equivariant with respect to $\symm_{nm}$.
\end{proof}

In particular, if $n$ is even, Lemma~\ref{lem:plethysm-interpretation} implies that the Frobenius image of the action of $\symm_n$ on the set $\PM_n$ of perfect matchings on $[n]$ is $s_{n/2}[s_2]$. It is an extremely difficult open problem to determine the $s$-expansion $s_a[s_b]$ for arbitrary $a,b$. However, if $n$ is even it is known that
\begin{equation}\label{eq:ple-exp}
    s_{n/2}[s_2] = \sum_{\substack{\lambda \, \vdash \, n \\ \text{$\lambda$ even}}} s_\lambda
\end{equation}
is the multiplicity-free sum over all $s_\lambda$ where $\lambda \vdash n$ has all even parts.

We will need a standard result on symmetrizers acting on $\symm_n$-irreducibles. Let $j \leq n$ and consider the embedding $\symm_j \subseteq \symm_n$ where $\symm_j$ acts on the first $j$ letters. This induces an embedding of group algebras $\CC[\symm_j] \subseteq \CC[\symm_n]$. We let $\eta_j \in \CC[\symm_j] \subseteq \CC[\symm_n]$ be the element which symmetrizes with respect to $\symm_j$, i.e.
\begin{equation}
   \eta_j := \sum_{w \, \in \, \symm_j} w.
\end{equation}
If $V$ is an $\symm_n$-module, the element $\eta_j$ acts as an operator on $V$. The following result characterizes when $\eta_j$ annihilates irreducible $\symm_n$-modules.

\begin{lemma}
    \label{lem:eta-annihilation}
    Let $j \leq n$. For a partition $\lambda \vdash n$, we have $\eta_j \cdot V^\lambda \neq 0$ if and only if $\lambda_1 \geq j$.
\end{lemma}

Lemma~\ref{lem:eta-annihilation} is standard, but we include a proof for the convenience of the reader.

\begin{proof}
   For any $\symm_n$-module $V$, it is not hard to check that
    \begin{equation}
        \eta_j \cdot V = V^{\symm_j} := \{ v \in V \,:\, w \cdot v = v \text{ for all $w \in \symm_j$} \}.
   \end{equation}
    Therefore, we have 
    \begin{equation}
        \label{eq:eta-equivalence}
        \eta_j \cdot V^{\lambda} \neq 0 \quad \Leftrightarrow \quad \left(\Res^{\symm_n}_{\symm_j}(V^{\lambda})\right)^{\symm_j} \neq 0 
    \end{equation}
    which is true if and only if the trivial representation $\one_{\symm_j}$ occurs with positive multiplicity in the restriction $\Res^{\symm_n}_{\symm_j}(V^{\lambda})$. The Branching Rule for $\symm_n$-modules (see e.g. \cite{Macdonald, Sagan}) states that $\Res^{\symm_n}_{\symm_{n-1}} V^\lambda = \bigoplus_\mu V^\mu$ where the direct sum is over partitions $\mu \vdash n-1$ obtained by removing an outer corner of $\lambda$. Iterating, we see that \eqref{eq:eta-equivalence} holds if and only if $\lambda_1 \geq j$.
\end{proof}

\section{The Matching Locus}
\label{sec:Matching}

As in the introduction, we consider the locus $\MMM_n = \{  w \in \symm_n \,:\, w^2 = 1 \} \subseteq \Mat_{n \times n}(\CC)$ of $n \times n$ permutation matrices corresponding to involutions. The symmetric group $\symm_n$ acts on $\MMM_n$ by conjugation: 
\begin{equation}
    v \cdot w := v w v^{-1} \quad \quad v \in \symm_n, \, w \in \MMM_n.
\end{equation}
The corresponding action of $\symm_n$ on $\CC[\xxx_{n \times n}]$ is 
\begin{equation}
    v \cdot x_{i,j} := x_{v(i),v(j)} \quad \quad v \in \symm_n, \, 1 \leq i,j \leq n.
\end{equation}
The ideal $\gr \, \II(\MMM_n) \subseteq \CC[\xxx_{n \times n}]$ is homogeneous and $\symm_n$-stable.

\subsection{Matching monomial basis}
In order to study the ring $R(\MMM_n)$, it will be useful to have an explicit generating set of its defining ideal $\gr \, \II(\MMM_n)$. It will turn out that such a generating set is as follows.

\begin{defn}
    \label{def:matching-ideal}
    Let $I_n^\MMM \subseteq \CC[\xxx_{n \times n}]$ be the ideal generated by 
    \begin{itemize}
        \item all sums $x_{i,1} + \cdots + x_{i,n}$ of variables in a single row,
        \item all sums $x_{1,j} + \cdots + x_{n,j}$ of variables in a single column,
        \item all products $x_{i,j} \cdot x_{i,j'}$ of variables in a single row,
        \item all products $x_{i,j} \cdot x_{i',j}$ of variables in a single column, and
        \item all diagonally symmetric differences $x_{i,j} - x_{j,i}$ of variables. 
    \end{itemize}
    Here $1 \leq i, i', j, j' \leq n$.
\end{defn}

We aim to show that $\gr \, \II(\MMM_n) = I^\MMM_n$ as ideals in $\CC[\xxx_{n \times n}]$. One of these containments is straightforward.

\begin{lemma}
    \label{lem:matching-ideal-containment}
    We have $I^\MMM_n \subseteq \gr \, \II(\MMM_n)$.
\end{lemma}

\begin{proof}
    As $\MMM_n=\{P(w) : w \in \symm_n, w^2=1\}$, it consists of all $n \times n$ symmetric $0$ and $1$ matrices with exactly one $1$ in each row and column. For $1 \leq i,i',j,j' \leq n$, the following polynomials are contained in (and in fact generate) the ideal $\II(\MMM_n)$:
    \begin{itemize}
        \item $x_{i,j}(x_{i,j}-1)$
        \item $x_{i,1} + \cdots + x_{i,n}-1$,
        \item $x_{1,j} + \cdots + x_{1,j}-1$,
        \item $x_{i,j} \cdot x_{i,j'}$,
        \item $x_{i,j} \cdot x_{i',j}$,
        \item $x_{i,j} - x_{j,i}$.
    \end{itemize}
    Note that the top degree homogeneous part of these polynomials are exactly the generators of $I^\MMM_n$, so the proof is complete.
\end{proof}

As explained in Remark~\ref{rmk:gr-generation}, the top degree components of a given generating set of $\II(\Zpoints)$ do not in general suffice to generate $\gr \, \II(\Zpoints)$. In order to show that we have generating in the context of Lemma~\ref{lem:matching-ideal-containment}, we show that the quotient $\CC[\xxx_{n \times n}]/I^\MMM_n$ has vector space dimension at most $|\MMM_n|$. Recall from the introduction that if $w \in \MMM_n$ is a matching, the matching monomial $\mmm(w) \in \CC[\xxx_{n \times n}]$ is the product over all variables $x_{i,j}$ for which $i < w(i) = j$.

\begin{theorem}
    \label{thm:matching-monomial-basis}
    The set $\{ \mmm(w) \,:\, w \in \MMM_n \}$ of matching monomials descends to a vector space basis of $R(\MMM_n)$.
\end{theorem}

\begin{proof}
    We show that $\{ \mmm(w) \,:\, w \in \MMM_n \}$ descends to a spanning set of the vector space $\CC[\xxx_{n \times n}]/I^\MMM_n$. This will establish that
    \begin{equation}
    \label{eq:matching-inequality-chain}
        |\MMM_n| = \dim R(\MMM_n) = \dim \CC[\xxx_{n \times n}]/\gr \, \II(\MMM_n) \leq \dim \CC[\xxx_{n \times n}]/I^\MMM_n \leq |\MMM_n|
    \end{equation}
    where the first equality follows from the orbit harmonics isomorphism \eqref{eq:orbit-harmonics-isomorphisms}, the second equality is the definition of $R(\MMM_n)$, the following inequality is a consequence of Lemma~\ref{lem:matching-ideal-containment}, and the last inequality is our spanning claim. This chain of (in)equalities will then show that $R(\MMM_n) = \CC[\xxx_{n \times n}]/I^\MMM_n$ and that $\{ \mmm(w) \,:\, w \in \MMM_n \}$ descends to a basis of $R(\MMM_n)$.

    It remains to show that $\{ \mmm(w) \,:\, w \in \MMM_n \}$ descends to a spanning set of $\CC[\xxx_{n \times n}]/I^\MMM_n$. Consider a monomial $m$ in the variable set $\xxx_{n \times n}$. We want to show that $m$ lies in the span of $\{ \mmm(w) \,:\, w \in \MMM_n \}$ modulo $I^\MMM_n$. 
    Write
    \begin{equation}
        m = x_{i_1,j_1}x_{i_2,j_2} \dots x_{i_k,j_k}
    \end{equation}
    for some $1 \leq i_1, j_1, i_2, j_2, \dots, i_k, j_k \leq n$. Since $x_{i,j}^2 \in I^\MMM_n$ and $x_{i,j} \equiv x_{j,i} \mod I^\MMM_n$ for all $1 \leq i, j \leq n$ we may assume that 
    \begin{quote}
        $(\triangle)$ the monomial $m$ is squarefree, and every variable $x_{i_s,j_s}$ which appears in $m$ satisfies $i_s \leq j_s$.
    \end{quote}
    We call a monomial in $\xxx_{n \times n}$ satisfying this condition {\em triangular}.
    We induct on the number $0 \leq r \leq k$ of factors $x_{i_s,j_s}$ of the triangular monomial $m$ which satisfy $i_s = j_s$.

    First assume that $r = 0$, so that $i_s \neq j_s$ for all $1 \leq s \leq k$. If there exist $1 \leq l \neq l' \leq k$ with $i_l = i_{l'}$ then $x_{i_l,j_l} \cdot x_{i_{l'},j_{j'}} \in I^\MMM_n$ so that $m \in I^\MMM_n$ certainly lies in the span of $\{ \mmm(w) \,:\, w \in \MMM_n \}$ modulo $I^\MMM_n$. Similarly, we are done if there exist $1 \leq l \neq l' \leq k$ with $j_l = j_{l'}$.
    If there exist $1 \leq l,l' \leq k$ with $i_l = j_{l'}$, then using $x_{i,j} \equiv x_{j,i} \mod I^\MMM_n$, we have 
    \begin{equation}
        x_{i_l,j_l} x_{i_{l'},j_{l'}} \equiv x_{j_l,i_l}x_{i_{l'},j_{l'}} \equiv 0 \mod I^\MMM_n,
    \end{equation}
    so that $m \in I^\MMM_n$ in this case, as well. We may therefore assume that the indices $i_1, j_1, i_2, j_2, \dots, i_k, j_k$ are distinct so that 
    $w = (i_1,j_1)(i_2,j_2) \dots (i_k,j_k)$ is the reduced cycle form of an involution  $w \in \symm_n$. By definition, we have
    $m = \mmm(w)$ which completes the base case.

    If $r > 0$ the triangular monomial $m$ contains the variable $x_{i,i}$ for some $i$. We may use the congruences
    \begin{equation}
        x_{i,i} \equiv - \sum_{\substack{j \, \neq \, i \\ 1 \, \leq \, j \, \leq \, n}} x_{i,j} \equiv - \sum_{\substack{j \, < \, i \\ 1 \, \leq \, j \, \leq \, n}} x_{j,i} - \sum_{\substack{j \, > \, i \\ 1 \, \leq \, j \, \leq \, n}} x_{i,j} \mod I^\MMM_n,
    \end{equation}
    to rewrite $m$ modulo $I^\MMM_n$ as the summation of triangular monomials of the same degree, but with one fewer variable on the main diagonal. By induction, we see that $m$ lies in the span of $\{ \mmm(w) \,:\, w \in \MMM_n \}$.
\end{proof}

For example, if $n = 4$ we have the matching monomial basis of $R(\MMM_4)$ given by
$$\{  1 \quad \mid \quad x_{1,2}, \, \, x_{1,3}, \, \, x_{1,4}, \, \, x_{2,3}, \, \,x_{2,4}, \, \, x_{3,4} \quad \mid \quad x_{1,2}  x_{3,4}, \, \, x_{1,3}  x_{2,4}, \, \,  x_{1,4}  x_{2,3} \}$$
where the bars indicate separation by degree.
We record the generating set of $\gr \, \II(\MMM_n)$ derived obtained in the above proof. It is precisely the set of highest degree components of the `natural' generating set of $\II(\MMM_n)$ shown in the proof of Lemma~\ref{lem:matching-ideal-containment}.

\begin{proposition}
    \label{prop:matching-ideal-equality}
    We have $I^\MMM_n = \gr \, \II(\MMM_n)$.
\end{proposition}
\begin{proof}
    In the proof of Theorem~\ref{thm:matching-monomial-basis}, we established Equation~\ref{eq:matching-inequality-chain}. Note that $|\MMM_n|$ on both ends of the chain forces all inequalities to be equalities. Specifically, we have  
    \begin{equation}
        \dim \CC[\xxx_{n \times n}]/\gr \, \II(\MMM_n) = \dim \CC[\xxx_{n \times n}]/I^\MMM_n.
    \end{equation}
    This together with Lemma~\ref{lem:matching-ideal-containment} gives the equality of the two ideals.
\end{proof}

\subsection{Module structure}
The basis of $R(\MMM_n)$ in Theorem~\ref{thm:matching-monomial-basis} is closed under the action $w \cdot x_{i,j} = x_{w(i), w(j)}$ of $\symm_n$. This is not a common property for `nice' bases of quotient rings with group actions. Thanks to this special property, we may easily compute the graded $\symm_n$-structure of $R(\MMM_n)$.

\begin{theorem}
    \label{thm:matching-frobenius}
    The graded Frobenius image of $R(\MMM_n)$ is given by
    $$\grFrob(R(\MMM_n); q) = \sum_{k \, = \, 0}^{\lfloor n/2 \rfloor} q^k \cdot s_k[s_2] \cdot s_{n-2k}.$$
\end{theorem}

\begin{proof}
    By Theorem~\ref{thm:matching-monomial-basis}, the degree $k$ piece $R(\MMM_n)_k$ of $R(\MMM_n)$ is isomorphic to the permutation action of $\symm_n$ on the vector space 
    \begin{equation}
        V_{n,k} := \mathrm{span}_\CC \{ \mu \text{ a matching on $[n]$} \,:\, \mu \text{ has exactly $k$ matched pairs} \}.
    \end{equation}
    It is not hard to see that $V_{n,k}$ decomposes as the induction product
    \begin{equation}
        V_{n,k} \cong \Ind_{\symm_{2k} \times \symm_{n-2k}}^{\symm_n} (V_{2k,k} \otimes V_{n-2k,0}) = V_{2k,k} \circ V_{n-2k,0} = V_{2k,k} \circ \one_{\symm_{n-2k}}.
    \end{equation}
    Lemma~\ref{lem:plethysm-interpretation} implies that $\Frob(V_{2k,k}) = s_k[s_2]$. Since $\Frob(\one_{\symm_{n-2k}}) = s_{n-2k}$, we are done.
\end{proof}

The explicit graded decomposition of $R(\MMM_n)$ into irreducibles is easily obtained from Equation~\eqref{eq:ple-exp}, Theorem~\ref{thm:matching-frobenius}, and the {\em Pieri Rule}
\begin{equation}
\label{eq:pieri-rule}
    s_\lambda \cdot s_b = \sum_{\nu} s_\nu
\end{equation}
where the sum is over partitions $\nu$ of such that $\lambda \subseteq \nu$ and the difference $\nu - \lambda$ of Young diagrams consists of $b$ boxes, no two of which share a column. For example, if $n = 6$ we have
\begin{small}
\begin{align*}
    \grFrob(\MMM_6;q) &= q^0 \cdot s_6 + q^1 \cdot s_1[s_2] \cdot s_4 + q^2 \cdot s_2[s_2] \cdot s_2 + q^3 \cdot s_3[s_2] \\
    &= q^0 \cdot s_6 + q^1 \cdot s_2 \cdot s_4 + q^2 \cdot (s_4 + s_{22}) \cdot s_2 + q^3 \cdot (s_6 + s_{42} + s_{222}) \\
    &= q^0 \cdot s_6 + q^1 \cdot (s_6 + s_{51} + s_{42}) + q^2 \cdot (s_6 + s_{51} + 2 s_{42} + s_{321} + s_{222})  + q^4 \cdot (s_6 + s_{42} + s_{222}).
\end{align*}
\end{small}
The Hilbert series of $R(\MMM_n)$ is an easy consequence of Theorem~\ref{thm:matching-monomial-basis}. Recall the double factorial $(2d-1)!! := (2d-1) \cdot (2d - 3) \cdot \cdots \cdot 1$.

\begin{corollary}
    \label{cor:matching-hilb}
    The Hilbert series of $R(\MMM_n)$ is given by
    $$\Hilb(R(\MMM_n);q) = \sum_{d \, = \, 0}^{\lfloor n/2 \rfloor} {n \choose 2d} \cdot (2d-1)!! \cdot q^d.$$
\end{corollary}

\begin{proof}
    Apply Theorem~\ref{thm:matching-monomial-basis} and count matching monomials by degree.
\end{proof}

In Figure~\ref{fig:matching-hilbert-series}, we give the histogram of $\Hilb(R(\MMM_n);q)$ when $n=200$. The horizontal axis corresponds to degrees $d$, and the height of each bar represents the corresponding $\dim R(\MMM_n)_d$.

\begin{figure}
    \centering
    \includegraphics[width=10cm]{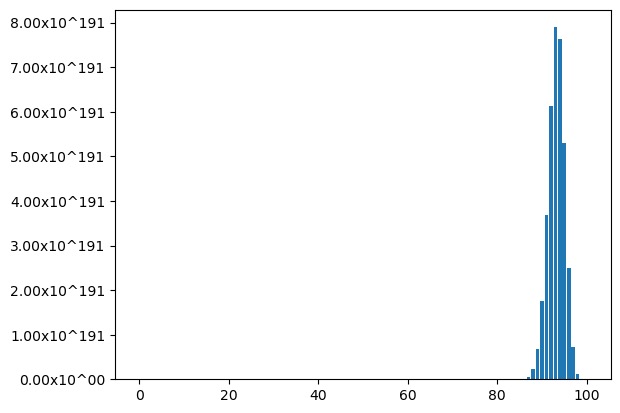}
    \caption{Coefficients of $\Hilb(R(\MMM_n);q)$ when $n=200$}
    \label{fig:matching-hilbert-series}
\end{figure}

\section{The Perfect Matching Locus}
\label{sec:Perfect}
In this section, assume $n$ is even. We consider the locus
$$\PM_n = \{ w \in \symm_n \,:\, w^2 = 1, \text{ $w$ has no fixed points} \}$$
of perfect matchings (or fixed-point-free involutions) inside $\MMM_n$. 
Our study of $R(\PM_n)$ is more indirect than our study of $R(\MMM_n)$. In particular, the authors do not know an explicit basis of $R(\PM_n)$. 

\subsection{$\eta$-annihilation}
As in the case of $\MMM_n$, it will be useful to understand the defining ideal $\gr \, \II(\PM_n)$ of $R(\PM_n)$. This ideal will turn out to have the following generating set.
\begin{defn}
    \label{def:pm-ideal}
    Assume $n > 0$ is even.
    Let $I^{\PM}_n \subseteq \CC[\xxx_{n \times n}]$ be the ideal
    $$I^{\PM}_n := I^\MMM_n + ( x_{i,i} \,:\, 1 \leq i \leq n) = \gr \, \II(\MMM_n) + (x_{i,i} \,:\, 1 \leq i \leq n).$$
\end{defn}

We will show that $I^\PM_n = \gr \, \II(\PM_n)$ as ideals in $\CC[\xxx_{n \times n}]$. As in the case of $\MMM_n$, one containment is straightforward.

\begin{lemma}
    \label{lem:pm-ideal-containment}
    For $n > 0$ even we have $I^{\PM}_n \subseteq \gr \, \II(\PM_n)$.
\end{lemma}
\begin{proof}
    Since $\PM_n\subseteq \MMM_n$ we have $\II(\MMM_n) \subseteq \II(\PM_n)$ and therefore  $\gr \, \II(\MMM_n)\subseteq \gr \, \II(\PM_n)$. Furthermore, we have  $x_{i,i}\in\II(\PM_n)$ for $1 \leq i \leq n$ since the involutions in $\PM_n$ have no fixed points. 
\end{proof}

Theorem~\ref{thm:matching-monomial-basis} gives bounds on the degrees of the quotient ring $\CC[\xxx_{n \times n}]/I_n^{\PM}$ by the ideal in Definition~\ref{def:pm-ideal}. Recall that $\eta_j = \sum_{w \in \symm_j} w \in \CC[\symm_n]$ acts on $\symm_n$-modules by symmetrization with respect to the first $j$ letters.

\begin{lemma}
    \label{lem:pm-symmetrizer-annihilation} Let $0 \leq j \leq n$ and assume $n$ is even.
   The symmetrizer $\eta_j \in \CC[\symm_j] \subseteq \CC[\symm_n]$ annihilates the degree $d$ part $(\CC[\xxx_{n \times n}]/I_n^\PM)_d$ of the ring $\CC[\xxx_{n \times n}]/I_n^\PM$ whenever $j > n-2d$. In particular, if $\lambda \vdash n$ and an irreducible $\symm_n$-module $V^\lambda$ appears in $(\CC[\xxx_{n \times n}]/I_n^\PM)_d$, we have $\lambda_1 \leq n - 2d$.
\end{lemma}
\begin{proof}
    The second statement follows from the first part and Lemma~\ref{lem:eta-annihilation}. We prove the first statement as follows.
    
    Since $I^\MMM_n\subseteq I^\PM_n$, by Theorem~\ref{thm:matching-monomial-basis} the set $\{\mmm(w):\text{$w\in\MMM_n$ has exactly $d$ $2$-cycles}\}$ of matching monoimals descends to a spanning set of $(\CC[\xxx_{n \times n}]/I_n^\PM)_d$. Therefore it suffices to show that \[\eta_j\cdot\mmm(w)\equiv 0 \mod{I_n^\PM}\] where $w\in\MMM_n$ has $d$ $2$-cycles and $j>n-2d$. 
    
    Write $\mmm(w)=\prod_{k=1}^{d}x_{i_k,j_k}$ and let $S=\{i_1,j_1,\dots,i_d,j_d\},T=S\cap [j]$. We have that
    \begin{equation}
    \label{eq:preliminary-eta-equality}
        \eta_j\cdot\mmm(w)\doteq\sum_{f:T\rightarrow [j]}\prod_{k=1}^{d}x_{f(i_k),f(j_k)}\mod{I_n^\PM},
    \end{equation} where $\doteq$ denotes equality up to a scalar multiple, and the summation is over all functions $f: T \to [j]$  with the convention  $f(\ell)=\ell$ if $\ell\not\in T$. The verification of \eqref{eq:preliminary-eta-equality} uses the fact that each diagonal variable $x_{i,i}$, each row product $x_{i,j}\cdot x_{i,j^{\prime}}$, and each column product $x_{i,j}\cdot x_{i^{\prime},j}$ lie in $I_n^\PM$.
    We want to show that 
    \begin{equation}\label{eq:PM-vanish-1}
      \sum_{f:T\rightarrow [j]}\prod_{k=1}^{d}x_{f(i_k),f(j_k)}\equiv 0 \mod{I_n^\PM}.
    \end{equation} 
    Let $T_k=\{i_k,j_k\}\cap [j]$. We calculate
     \begin{multline}\label{eq:pm-ann-help}
        \sum_{f:T\rightarrow [j]}\prod_{k=1}^{d}x_{f(i_k),f(j_k)}=\prod_{k=1}^{d}\bigg(\sum_{g:T_k\rightarrow [j]}x_{g(i_k),g(j_k)}\bigg)\doteq\prod_{k=1}^{d}\bigg(\sum_{g:T_k\rightarrow [n]\setminus [j]}x_{g(i_k),g(j_k)}\bigg)\\
        =\sum_{f:T\rightarrow [n]\setminus [j]} \left(\prod_{k=1}^{d}x_{f(i_k),f(j_k)} \right)\equiv\sum_{f:T\rightarrow ([n]\setminus [j])\setminus (S\setminus T)} \left( \prod_{k=1}^{d}x_{f(i_k),f(j_k)} \right)\mod{I^\PM_n}
     \end{multline} with the convention  $g(\ell)=\ell$ if $\ell\not\in T_k$, where the $\doteq$ is true modulo $I^\PM_n$ because each row sum $\sum_{j=1}^{n}x_{i,j}$ and each column sum $\sum_{i=1}^{n}x_{i,j}$ lies in $I^\PM_n$ and the congruence $\equiv$ holds because each row product and column product lies in $I^{\PM}_n$.

     In order to prove Equation~\eqref{eq:PM-vanish-1}, it remains to show 
     \begin{equation}\label{eq:PM-vanish-2}
         \sum_{f:T\rightarrow ([n]\setminus [j])\setminus (S\setminus T)} \left( \prod_{k=1}^{d}x_{f(i_k),f(j_k)} \right) \equiv 0\mod{I^\PM_n}.
     \end{equation}
     Note that the set $([n]\setminus [j])\setminus (S\setminus T)$ has cardinality 
    \begin{equation}
     \lvert([n]\setminus [j])\setminus (S\setminus T)\rvert=n-j-(2d-\lvert T\rvert)=n-j-2d+\lvert T\rvert <\lvert T\rvert.
     \end{equation}
     Every function $f:T\rightarrow ([n]\setminus [j])\setminus (S\setminus T)$ indexing the sum in \eqref{eq:PM-vanish-2} is therefore {\bf not} an injection. This given, Equation~\eqref{eq:PM-vanish-2} follows from the fact that each diagonal variable $x_{i,i}$, each row product $x_{i,j}\cdot x_{i,j^{\prime}}$ and each column product $x_{i,j}\cdot x_{i^{\prime},j}$ lie in $I_n^\PM$.
\end{proof}

\subsection{Module structure}
Lemma~\ref{lem:pm-symmetrizer-annihilation} gives a bound on the irreducible representations appearing in $\CC[\xxx_{n \times n}]/I^\PM_n$. This bound allows us to deduce the graded Frobenius image of $R(\PM_n)$ from that of $R(\MMM_n)$.

\begin{theorem}
    \label{thm:pm-frobenius}
    For $n > 0$ even, the graded Frobenius image of $R(\PM_n)$ is given by
   $$\grFrob(R(\PM_n);q) =  \sum_{\substack{\lambda \, \vdash \, n \\ \text{$\lambda$ even}}} q^{\frac{n-\lambda_1}{2}}\cdot s_\lambda.$$
\end{theorem}

\begin{proof}
    Definition~\ref{def:pm-ideal} and Lemma~\ref{lem:pm-ideal-containment} imply that $I^\MMM_n\subseteq I^\PM_n \subseteq \gr \, \II(\PM_n)$. Therefore, we have graded $\symm_n$-equivariant surjections
\begin{equation}\label{eq:pm-surj}   
   R(\MMM_n)\twoheadrightarrow \CC[\xxx_{n \times n}]/I_n^\PM \twoheadrightarrow R(\PM_n).
   \end{equation} 
    By Theorem~\ref{thm:matching-frobenius} and the Pieri Rule, any irreducible $V^\lambda$ appearing in $R(\MMM_n)_d$ will have $\lambda_1 \geq n-2d$. Thus, any $V^\lambda $ appearing in $(\CC[\xxx_{n \times n}]/I^\PM_n)_d$ and $R(\PM_n)_d$ will also satisfy $\lambda_1 \geq n-2d$. On the other hand, Lemma~\ref{lem:pm-symmetrizer-annihilation} says that each  $V^{\lambda}$ appearing in $(\CC[\xxx_{n \times n}]/I_n^\PM)_d$ satisfies $\lambda_1\le n-2d$, hence each Specht module $V^{\lambda}$ appearing in $R(\PM_n)_d$ should also satisfy $\lambda_1 \leq n-2d$. It follows that 
    \begin{quote} {\em every irreducible $V^\lambda$ appearing in $R(\PM_n)_d$ satisfies $\lambda_1 = n-2d$.}
    \end{quote}

    Lemma~\ref{lem:plethysm-interpretation} and Equation~\eqref{eq:ple-exp} yield the ungraded Frobenius image
    \begin{equation}
        \Frob(R(\PM_n)) = \Frob(\CC[\PM_n]) = s_{n/2}[s_2]= \sum_{\substack{\lambda \, \vdash \, n \\ \text{$\lambda$ even}}} s_\lambda.
    \end{equation}
   Since $R(\PM_n) = \bigoplus_{d \geq 0} R(\PM_n)_d$,
    the italicized statement in the last paragraph forces
    \begin{equation}
        \Frob(R(\PM_n)_d) =\sum_{\substack{\lambda \, \vdash \, n \\ \text{$\lambda$ even} \\ \lambda_1 \,= \,n-2d}}  s_\lambda,
    \end{equation}
    which completes the proof.
    \end{proof}

    The reasoning in Theorem~\ref{thm:pm-frobenius} did not directly show that $I^{\PM}_n = \gr \, \II(\PM_n)$. We deduce this equality of ideals using some symmetric function theory.

\begin{proposition}
    \label{prop:pm-ideal-equality}
    We have $I^{\PM}_n = \gr \, \II(\PM_n)$.
\end{proposition}
\begin{proof}
Recall the symmetric function operators $\{ - \}_{\lambda_1 \leq n-2d}$ and $\{ - \}_{\lambda_1 = n-2d}$ introduced in Section~\ref{sec:Background}.
For any $d \leq n/2$, we claim that 
\begin{equation}
\label{eq:inequality-pm}
        \left\{ s_d[s_2]\cdot s_{n-2d}\right\}_{\lambda_1\le n-2d} \, \, \le  \,  \sum_{\substack{\lambda \, \vdash \, n \\ \text{$\lambda$ even} \\ \lambda_1 \, = \, n-2d}}s_{\lambda}= \{s_{n/2}[s_2]\}_{\lambda_1  = n-2d}
\end{equation}
where $F \leq G$ means that $G-F$ is Schur-positive.
Indeed, by Pieri's Rule and Equation~\eqref{eq:ple-exp},
   \begin{equation}
    \label{eq:plethym-pieri}
        \left\{ s_d[s_2]\cdot s_{n-2d} \right\}_{\lambda_1\le n-2d} \, \, = \,  \sum_{(\mu,\nu)}s_{\lambda}
    \end{equation}
    where the sum is over all ordered pairs $(\mu,\nu)$ consisting of 
    \begin{itemize}
       \item an even partition $\mu\vdash 2d$, and
\item a horizontal strip $\nu$ of size $n-2d$ where $\lambda/\mu = \nu$, such that
        \item $\lambda_1 \leq n-2d$
   \end{itemize}  
   Since $|\nu| = n-2d$, the Pieri Rule forces $\lambda_1\ge n-2d$, so that each $s_\lambda$ appearing in \eqref{eq:plethym-pieri} satisfies $\lambda_1=n-2d=\lvert\nu\rvert$. Thus each column of $\lambda$ contains a box of $\nu$, which means that each part of $\lambda$ equals some part of $\mu$ or $n-2d$. Therefore $\lambda$ is even and $(\mu,\nu)$ is uniquely determined by $\lambda$. Hence the claimed inequality \eqref{eq:inequality-pm} holds.

    We deduce the proposition from \eqref{eq:inequality-pm}. If $V$ is a graded $\symm_n$-module, write $V_{\leq d}$ for the restriction of $V$ to degrees $\leq d$. We have 
    \begin{equation}\label{eq:pm-surj-refine}
        \sum_{d \, = \, 0}^{\frac{n}{2}-1}q^d\cdot \{s_d[s_2]\cdot s_{n-2d}\}_{\lambda_1\le n-2d}  \,\,
        \ge \, \, \grFrob(\CC[\xxx_{n \times n}]/I_n^\PM;q) \,\, \ge  \, \,\grFrob(R(\PM_n);q).
    \end{equation}
   where 
    \begin{itemize}
       \item the first inequality in \eqref{eq:pm-surj-refine} follows from the ideal containment $\gr \, \II(\MMM_n) = I^{\MMM}_n \subseteq I^{\PM}_n$, Theorem~\ref{thm:matching-frobenius}, and Lemma~\ref{lem:pm-symmetrizer-annihilation}, and
       \item the second inequality in \eqref{eq:pm-surj-refine} follows from Lemma~\ref{lem:pm-ideal-containment}.
    \end{itemize}
     On the other hand, by \eqref{eq:inequality-pm} and Theorem~\ref{thm:pm-frobenius} we have 
    \begin{equation}
   \sum_{d \, = \, 0}^{\frac{n}{2}-1} \{s_d[s_2]\cdot s_{n-2d}\}_{\lambda_1\le n-2d} \, \, \le \,\, \sum_{d\, =\, 0}^{\frac{n}{2}-1}\{s_{n/2}[s_2]\}_{\lambda_1=n-2d}=s_{n/2}[s_2]=\Frob(R(\PM_n)),
    \end{equation}
    which implies that \eqref{eq:pm-surj-refine} is a chain of equalities. Lemma~\ref{lem:pm-ideal-containment} implies $I^{\PM}_n \subseteq \gr \, \II(\PM_n)$, and the proposition follows. 
\end{proof}

Theorem~\ref{thm:matching-frobenius} gives a combinatorial formula for the Hilbert series of $R(\PM_n)$. If $w \in \symm_n$ is a permutation, a {\em decreasing subsequence of length $k$}  is a sequence $1 \leq i_1 < \cdots < i_k \leq n$ of indices such that $w(i_1) > \cdots > w(i_k)$. We write
\begin{equation}
    \lds(w) := \text{longest decreasing subsequence of $w$}
\end{equation}
for the longest possible length of a decreasing subsequence of $w$. It is a famous result of Baik, Deift, and Johansson \cite{BDJ} that the distribution of the statistic $\lds$ on $\symm_n$ converges to the Tracy-Wildom distribution, after appropriate rescaling. Restricting $\lds$ to the set $\PM_n$ of fixed-point-free involutions gives the Hilbert series of $R(\PM_n)$, up to reversal.

\begin{corollary}
    \label{cor:pm-hilbert}
    Assume that $n$ is even. The Hilbert series of $R(\PM_n)$ is given by
   \begin{equation}
        \Hilb(R(\PM_n);q) = \sum_{w \, \in \, \PM_n} q^{\frac{n - \lds(w)}{2}}.
    \end{equation}
\end{corollary}

\begin{proof}
   We recall several standard facts about the Schensted correspondence. Suppose $w \in \symm_n$ satisfies $w \mapsto (P,Q)$.
    \begin{enumerate}
      \item We have $w^{-1} \mapsto (Q,P)$. In particular, we have $w \in \MMM_n$ if and only if $P = Q$.
       \item The statistic $\lds(w)$ equals the length of the first column of $P$ (or $Q$).
       \item If $w \in \MMM_n$, the number of fixed points of $w$ equals the number of columns in $P$ of odd length.
    \end{enumerate}
    From these facts, for any $d \leq n/2$ we have
   \begin{equation}
       | \{ w \in \PM_n \,:\, \lds(w) = n - 2d \} | = \sum_{\substack{\lambda \, \vdash \, n \\ \lambda'_1 \, = \, n-2d \\ \lambda' \text{ even}}} |\SYT(\lambda)|  =
        \sum_{\substack{\lambda \, \vdash \, n \\ \lambda_1 \, = \, n-2d \\ \lambda \text{ even}}} |\SYT(\lambda)|
    \end{equation}
   where $\lambda'$ is the conjugate partition obtained from interchanging the rows and columns of $\lambda$. Since $\dim V^\lambda = |\SYT(\lambda)|$, we are done by Theorem~\ref{thm:pm-frobenius}.
\end{proof}

The asymptotics of the distribution in Corollary~\ref{cor:pm-hilbert} may be derived from work of Baik and Rains \cite[Thm. 3.1]{BR}. After rescaling, the distribution of $\lds$ on $\MMM_n$ converges to the GOE distribution of random matrix theory as $n \to \infty$. Figure~\ref{fig:matching-hilbert-series} gives the histogram of $\Hilb(R(\PM_n);q)$ when $n=100$. The horizontal axis corresponds to degrees $d$, and the height of each bar represents the corresponding $\dim R(\PM_n)_d$.

\begin{figure}
    \centering
    \includegraphics[width=10cm]{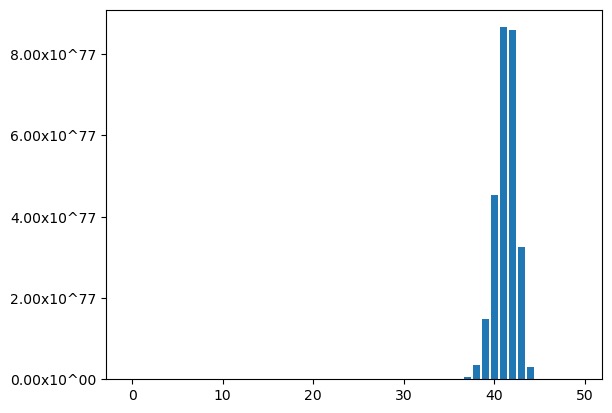}
    \caption{Coefficients of $\Hilb(R(\PM_n);q)$ when $n=100$} 
    \label{fig:PM-hilbert-series}
\end{figure}

\section{Conjugacy classes of involutions}
\label{sec:Conjugacy}

For $0 \leq a \leq n$ with $a \equiv n \mod 2$, we have the following subset of the locus $\MMM_n \subseteq \Mat_{n \times n}(\CC)$ of all involution permutation matrices: 
\begin{equation}
    \MMM_{n,a} := \{ w \in \MMM_n \,:\, \text{$w$ has $a$ fixed points} \}.
\end{equation}
The set of matrices $\MMM_{n,a}$ forms a single orbit under the conjugation action of $\symm_n$, and can be identified with permutations of cycle type $(2^{(n-a)/2},1^a)$. We have the disjoint union $\MMM_n = \bigsqcup_{a} \MMM_{n,a}$.
Our study of $R(\MMM_{n,a})$ will be more indirect and involved than our study of $R(\MMM_n)$ or $R(\PM_n)$; we will calculate the graded $\symm_n$-structure of $R(\MMM_{n,a})$ in Theorem~\ref{thm:conjugacy-module-character} after many lemmas. The added difficulties in studying $R(\MMM_{n,a})$ are the lack of an explicit basis for this quotient ring (as in Theorem~\ref{thm:matching-monomial-basis} for $R(\MMM_n)$) and the failure of $\eta$-operator annihilation to determine $R(\MMM_{n,a})$ (as it did in the proof of Theorem~\ref{thm:pm-frobenius} for $R(\PM_n)$).

\subsection{Some symmetric function results}
This subsection proves technical results on Schur expansions which will assist in our analysis of $R(\MMM_{n,a})$. This material should probably be skipped on a first reading and returned to as needed. 

These technical results require some notation. If $f(q) = \sum_{d} c_d \cdot q^d$ is a polynomial in $q$, we write $\langle q^d \rangle f := c_d$ for the coefficient of $q^d$ in $f$. Similarly, if $F = \sum_\lambda c_\lambda \cdot s_\lambda$ is a symmetric function expressed in the Schur basis, we write $\langle s_\lambda \rangle F := c_\lambda$ for the coefficient of $s_\lambda$. Finally, if $m \in \ZZ$ we write $m \mod 2 \in \{0,1\}$ for the remainder of $m$ modulo $2$. The next lemma describes the Schur expansion of the ungraded Frobenius image of $\CC[\MMM_{2d+a,d}]$.

\begin{lemma}
    \label{lem:s-coef}
    For each $\lambda\vdash a+2d$, write $\lambda=\lambda_1^{d_1}\lambda_2^{d_2}\cdots\lambda_m^{d_m}$ where $\lambda_1>\lambda_2>\cdots>\lambda_m>0$, and let $\delta_i := \lambda_i \mod 2 \in \{0,1\}$. Then
    \begin{equation}\label{eq:s-coef}
    \langle s_{\lambda} \rangle (s_d[s_2]\cdot s_a)=\left\langle q^{\frac{a-\sum_{i=1}^{m}\delta_i}{2}}\right\rangle \Bigg(\prod_{i=1}^{m}\sum_{j=0}^{\left\lfloor\frac{\lambda_i}{2}\right\rfloor-\left\lceil\frac{\lambda_{i+1}}{2}\right\rceil}q^j\Bigg)
    \end{equation}
    with the convention $\lambda_{m+1}=0$.
\end{lemma}
\begin{proof}
    By Pieri's Rule and Equation~\eqref{eq:ple-exp}, we have
    \begin{multline}
      \langle s_{\lambda} \rangle(s_d[s_2]\cdot s_a)
      =\left\lvert\left\{\text{horizontal strips $\mu \subseteq \lambda$}\,:\,
      \begin{matrix}\text{$\lambda/\mu$ is a partition,} \\  \text{$\lvert\mu\rvert=a$, and $\lambda/\mu$ is even}\end{matrix}
      \right\}\right\rvert
      = \\ \left\lvert\left\{(x_1,x_2,\cdots,x_m)\in\mathbb{Z}_{\ge 0}^m\,:\,\text{$x_i\le\lambda_i-\lambda_{i+1}$, $x_i\equiv\lambda_i\mod{2}$, and $\sum_{i=1}^{m}x_i=a$}\right\}\right\rvert.
    \end{multline}
    One the other hand, 
    \begin{multline}
    \left\langle q^{\frac{a-\sum_{i=1}^{m}\delta_i}{2}} \right \rangle \Bigg(\prod_{i=1}^{m}\sum_{j=0}^{\left\lfloor\frac{\lambda_i}{2}\right\rfloor-\left\lceil\frac{\lambda_{i+1}}{2}\right\rceil}q^j\Bigg)
    =  \\ \left\lvert\left\{(y_1,y_2,\cdots,y_m)\in\mathbb{Z}_{\ge 0}^m\,:\,\text{$y_i \le \left\lfloor\frac{\lambda_i}{2}\right\rfloor-\left\lceil\frac{\lambda_{i+1}}{2}\right\rceil$ and $\sum_{i=1}^{m}y_i=$}\frac{a-\sum_{i=1}^{m}\delta_i}{2}\right\}\right\rvert.
    \end{multline}
    We show that 
    \begin{multline}
    \label{eqn:coefficient-desired-interpretation}
    \left\lvert\left\{(x_1,x_2,\cdots,x_m)\in\mathbb{Z}_{\ge 0}^m\,:\,\text{$x_i\le\lambda_i-\lambda_{i+1}$, $x_i\equiv\lambda_i\mod{2}$, and $\sum_{i=1}^{m}x_i=a$}\right\}\right\rvert
    = \\ \left\lvert\left\{(y_1,y_2,\cdots,y_m)\in\mathbb{Z}_{\ge 0}^m\,:\,\text{$y_i \le \left\lfloor\frac{\lambda_i}{2}\right\rfloor-\left\lceil\frac{\lambda_{i+1}}{2}\right\rceil$ and $\sum_{i=1}^{m}y_i=\frac{a-\sum_{i=1}^{m}\delta_i}{2}$}\right\}\right\rvert.
    \end{multline}
    Indeed, if we let
    \begin{equation} A:=\left\{(x_1,x_2,\cdots,x_m)\in\mathbb{Z}_{\ge 0}^m\,:\,\text{$x_i\le\lambda_i-\lambda_{i+1}$, $x_i\equiv\lambda_i\mod{2}$, and $\sum_{i=1}^{m}x_i=a$}\right\}\end{equation}
    and
    \begin{equation}B:=\left\{(y_1,y_2,\cdots,y_m)\in\mathbb{Z}_{\ge 0}^m\,:\,\text{$y_i \le \left\lfloor\frac{\lambda_i}{2}\right\rfloor-\left\lceil\frac{\lambda_{i+1}}{2}\right\rceil$ and $\sum_{i=1}^{m}y_i=\frac{a-\sum_{i=1}^{m}\delta_i}{2}$}\right\},\end{equation}
    we have a map $f: A \to B$ given by
    \begin{equation}
        f:
        (x_i)_{i\in [m]}\longmapsto \Big(\left\lfloor\frac{x_i}{2}\right\rfloor\Big)_{i\in [m]}.
    \end{equation}
    It remains to show that $f$ is a bijection as follows.

    First, we show that $f: A \to B$ is well-defined, i.e. $f(A)\subseteq B$. In fact, if $(x_i)_{i\in [m]}\in A$, since $x_i\equiv\lambda_i\mod{2}$ we have 
    \begin{equation}
    \label{eqn:xi-to-deltai}
    \left\lfloor\frac{x_i}{2}\right\rfloor=\frac{x_i-\delta_i}{2}.\end{equation}
    Therefore 
    \begin{equation}\left\lfloor\frac{x_i}{2}\right\rfloor\le\frac{\lambda_i-\lambda_{i+1}-\delta_i}{2},\end{equation}
    which forces 
    \begin{equation}\left\lfloor\frac{x_i}{2}\right\rfloor\le\left\lfloor\frac{\lambda_i-\lambda_{i+1}-\delta_i}{2}\right\rfloor=\frac{\lambda_i-\lambda_{i+1}-\delta_i-\delta_{i+1}}{2}=\left\lfloor\frac{\lambda_i}{2}\right\rfloor-\left\lceil\frac{\lambda_{i+1}}{2}\right\rceil.\end{equation}
    Since 
    \begin{equation}\sum_{i\,=\,1}^m \left\lfloor\frac{x_i}{2}\right\rfloor=\sum_{i\,=\,1}^m \frac{x_i-\delta_i}{2}=\frac{a-\sum_{i=1}^m\delta_i}{2}\end{equation}
    we have $(\lfloor\frac{x_i}{2}\rfloor)_{i\in [m]}\in B$ so that $f: A \to B$ is well-defined.

    Next, we claim that $f: A \to B$ is injective. Indeed, for any tuple $(x_1, \dots, x_m) \in A$, Equation~\eqref{eqn:xi-to-deltai} implies that each entry $x_i$ is determined by $\left \lfloor \frac{x_i}{2} \right \rfloor$. The injectivity of $f$ follows.

    Finally, we show that $f$ is surjective. Let $(y_1, \dots, y_m)\in B$ and set $x_i:=2y_i+\delta_i$ for $1 \leq i \leq m$. We claim that $(x_1, \dots, x_m) \in A$. Given $1 \leq i \leq m$, since 
    \begin{equation}y_i \le \left\lfloor\frac{\lambda_i}{2}\right\rfloor-\left\lceil\frac{\lambda_{i+1}}{2}\right\rceil \end{equation}
    we have 
    \begin{equation}x_i\le 2\left(\left\lfloor\frac{\lambda_i}{2}\right\rfloor-\left\lceil\frac{\lambda_{i+1}}{2}\right\rceil\right)+\delta_i=2\left(\frac{\lambda_i-\delta_i}{2}-\frac{\lambda_{i+1}+\delta_{i+1}}{2}\right)+\delta_i=\lambda_i-\lambda_{i+1}-\delta_{i+1}\le \lambda_i-\lambda_{i+1}.\end{equation} 
    In addition, \begin{equation}x_i=2y_i+\delta_i\equiv\delta_i\equiv\lambda_i\mod{2}.\end{equation} 
    Furthermore, we have 
    \begin{equation}\sum_{i\,=\,1}^m x_i=2\sum_{i\,=\,1}^m y_i+\sum_{i\,=\,1}^m \delta_i=2\cdot\frac{a-\sum_{i=1}^m\delta_i}{2}+\sum_{i\,=\,1}^m\delta_i=a.\end{equation} 
    We conclude that $(x_1,\dots,x_m) \in A$, and we have $f: (x_1, \dots, x_m) \mapsto (y_1, \dots, y_m)$. Therefore, the map $f$ is a bijection, Equation~\eqref{eqn:coefficient-desired-interpretation} holds, and the proof is complete.
\end{proof}

Our second lemma is an identity of symmetric functions. Roughly speaking, it stratifies the Schur expansion of $s_{(n-a)/2}[s_2] \cdot s_a$ according to the length of the first row.

\begin{lemma}
    \label{lem:s-identity-one}
    If $a \equiv n \mod 2$ we have the identity of symmetric functions
    \begin{equation}\label{eq:s-identity-one}
        \sum_{d \, = \, 0}^{(n-a)/2} \{
            s_d[s_2] \cdot s_{n-2d} - s_{d-1}[s_2] \cdot s_{n-2d+2} 
        \}_{\lambda_1 \leq n - 2d + a} = s_{(n-a)/2}[s_2] \cdot s_a.
    \end{equation}
\end{lemma}
\begin{proof}
    For any  $a\equiv b\equiv n\mod{2}$, we claim that
    \begin{equation}\label{eq:a+b}
        \left\{s_{(n-a)/2}[s_2]\cdot s_a\right\}_{\lambda_1=a+b}=\left\{s_{(n-b)/2}[s_2]\cdot s_b\right\}_{\lambda_1=a+b},
    \end{equation}
    \begin{equation}\label{eq:a+b-1}
        \left\{s_{(n-a)/2}[s_2]\cdot s_a\right\}_{\lambda_1=a+b-1}=\left\{s_{(n-b)/2}[s_2]\cdot s_b\right\}_{\lambda_1=a+b-1}.
    \end{equation}
    For Equation~\eqref{eq:a+b}, fix $\lambda\vdash n$ such that $\lambda_1=a+b$ and write $\lambda=\lambda_1^{d_1}\lambda_2^{d_2}\cdots\lambda_m^{d_m}$ where $a+b=\lambda_1>\lambda_2>\cdots>\lambda_m>0$. Let $\delta_i := \lambda_i \mod 2 \in \{0,1\}$ for $1 \leq i \leq m$. Consider the polynomial $f \in \ZZ[q]$ given by
    \begin{equation}f:=\prod_{i=1}^{m}\sum_{j=0}^{\left\lfloor\frac{\lambda_i}{2}\right\rfloor-\left\lceil\frac{\lambda_{i+1}}{2}\right\rceil}q^j.\end{equation}
    Observe that the coefficient sequence of $f$ is palindromic.
    Since $\lambda_1=a+b\equiv 0\mod{2}$, we have $\delta_1=0$. Therefore 
    \begin{multline}\frac{a-\sum_{i=1}^m\delta_i}{2}+\frac{b-\sum_{i=1}^m\delta_i}{2}
    =\frac{a+b-\sum_{i=1}^m(\delta_i+\delta_{i+1})}{2}
    = \frac{\lambda_1-\sum_{i=1}^m(\delta_i+\delta_{i+1})}{2}
    = \\ \sum_{i\,=\,1}^m\left(\frac{\lambda_i-\delta_i}{2}-\frac{\lambda_{i+1}+\delta_{i+1}}{2}\right)
    =\sum_{i\,=\,1}^m\left(\left\lfloor\frac{\lambda_i}{2}\right\rfloor-\left\lceil\frac{\lambda_{i+1}}{2}\right\rceil\right)=\deg(f).
    \end{multline}
    The palindromicity of $f$ yields
    \begin{equation}
    \label{eq:a+b-help}
    \left\langle q^{\frac{a-\sum_{i=1}^m\delta_i}{2}} \right \rangle f=
    \left\langle q^{\frac{b-\sum_{i=1}^m\delta_i}{2}}\right\rangle f.
    \end{equation}
    Equation~\eqref{eq:a+b} follows from Equation~\eqref{eq:a+b-help} and Lemma~\ref{lem:s-coef}.
    The proof of Equation~\eqref{eq:a+b-1} is similar to that of Equation~\eqref{eq:a+b} and omitted.
    
    We use Equations~\eqref{eq:a+b} and \eqref{eq:a+b-1} to prove the desired Equation~\eqref{eq:s-identity-one}. The left hand side of Equation~\eqref{eq:s-identity-one} may be rewritten
    \begin{multline}
    \label{eq:s-identity-one-leftside}
        \sum_{d \, = \, 0}^{\frac{n-a}{2}} \{
            s_d[s_2] \cdot s_{n-2d} - s_{d-1}[s_2] \cdot s_{n-2d+2} 
        \}_{\lambda_1 \leq n - 2d + a} = \\
        \sum_{d \, = \, 0}^{\frac{n-a}{2}} \{
            s_d[s_2] \cdot s_{n-2d} 
        \}_{\lambda_1 \leq n - 2d + a}-\sum_{d \, = \, 0}^{\frac{n-a}{2}} \{
            s_{d-1}[s_2] \cdot s_{n-2d+2} 
        \}_{\lambda_1 \leq n - 2d + a} = \\
       \sum_{d \, = \, 0}^{\frac{n-a}{2}} \{
            s_d[s_2] \cdot s_{n-2d} 
        \}_{\lambda_1 \leq n - 2d + a}-\sum_{d \, = \, 0}^{\frac{n-a}{2}-1} \{
            s_{d}[s_2] \cdot s_{n-2d} 
        \}_{\lambda_1 \leq n - 2d + a-2} = \\
        \{s_{(n-a)/2}[s_2]\cdot s_a \}_{\lambda_1 \le 2a} + \sum_{d \, = \, 0}^{\frac{n-a}{2}-1} \{
            s_{d}[s_2] \cdot s_{n-2d} 
        \}_{n - 2d + a-1 \le \lambda_1 \le n-2d+a}.
    \end{multline}
    Using Equations~\eqref{eq:a+b} and \eqref{eq:a+b-1} with $b=n-2d$, we have 
    \begin{multline}\label{eq:s-identity-one-new}
        \{s_{(n-a)/2}[s_2]\cdot s_a \}_{\lambda_1 \le 2a} + \sum_{d \, = \, 0}^{\frac{n-a}{2}-1} \{
            s_{d}[s_2] \cdot s_{n-2d} 
        \}_{n - 2d + a-1 \le \lambda_1 \le n-2d+a} = \\
        \{s_{(n-a)/2}[s_2]\cdot s_a \}_{\lambda_1 \le 2a} + \sum_{d \, = \, 0}^{\frac{n-a}{2}-1} \{
            s_{(n-a)/2}[s_2] \cdot s_{a} 
        \}_{n - 2d + a-1 \le \lambda_1 \le n-2d+a}
        =s_{(n-a)/2}[s_2]\cdot s_a.
    \end{multline}
    The desired Equation~\eqref{eq:s-identity-one} follows from Equations~\eqref{eq:s-identity-one-leftside} and \eqref{eq:s-identity-one-new}.
\end{proof}

Our final symmetric function lemma is an equality \eqref{eq:s-identity-two} of two sums involving Schur function plethysms. The right hand side is the Frobenius image $\sum_d s_d[s_2] \cdot s_{n-2d}$ of the $\symm_n$-module $\CC[\MMM_n]$. The terms of the left hand side  involve the Frobenius images $s_d[s_2] \cdot s_{n-2d}$ of the submodules $\CC[\MMM_{n,n-2d}]$, but are truncated by the length of the first row.

\begin{lemma}
    \label{lem:s-identity-two}
    We have the identity of symmetric functions
    \begin{equation}\label{eq:s-identity-two}
        \sum_{d \, = \, 0}^{\lfloor n/2 \rfloor} \big( \{
            s_d[s_2] \cdot s_{n-2d} \}_{\lambda_1 \leq 2(n-2d)} + \{s_{d}[s_2] \cdot s_{n-2d} 
        \}_{\lambda_1 \leq 2(n-2d)-2} \big) = \sum_{d \, = \, 0}^{\lfloor n/2 \rfloor} s_d[s_2] \cdot s_{n-2d}.
    \end{equation}
\end{lemma}

The proof of Lemma~\ref{lem:s-identity-two} is an application of Equation~\eqref{eq:s-identity-one-new} to \eqref{eq:s-identity-two}. Thanks to these tools, this proof is not as difficult as that of Lemmas~\ref{lem:s-coef} or \ref{lem:s-identity-one}.

\begin{proof}
    Applying Equation~\eqref{eq:s-identity-one-new} to the right hand side of Equation~\eqref{eq:s-identity-two}, we get
    \begin{align}
        \notag \sum_{d \, = \, 0}^{\lfloor n/2 \rfloor} &s_d[s_2] \cdot s_{n-2d}= \\
       & \sum_{d \, = \, 0}^{\lfloor n/2 \rfloor} \Bigg( \{s_d[s_2]\cdot s_{n-2d}\}_{\lambda_1\le 2(n-2d)} + \sum_{k=0}^{d-1} \{s_k[s_2]\cdot s_{n-2k}\}_{2n-2k-2d-1\le\lambda_1\le 2n-2k-2d} \Bigg)=\\
       & \sum_{d \, = \, 0}^{\lfloor n/2 \rfloor} \{s_d[s_2]\cdot s_{n-2d}\}_{\lambda_1\le 2(n-2d)} + \sum_{d \, = \, 0}^{\lfloor n/2 \rfloor} \sum_{k=0}^{d-1} \{s_k[s_2]\cdot s_{n-2k}\}_{2n-2k-2d-1\le\lambda_1\le 2n-2k-2d}= \\
       & \sum_{d \, = \, 0}^{\lfloor n/2 \rfloor} \{s_d[s_2]\cdot s_{n-2d}\}_{\lambda_1\le 2(n-2d)} + \sum_{k \, = \, 0}^{\lfloor n/2 \rfloor-1} \sum_{d=k+1}^{\lfloor n/2 \rfloor} \{s_k[s_2]\cdot s_{n-2k}\}_{2n-2k-2d-1\le\lambda_1\le 2n-2k-2d}=\\
       & \sum_{d \, = \, 0}^{\lfloor n/2 \rfloor} \{s_d[s_2]\cdot s_{n-2d}\}_{\lambda_1\le 2(n-2d)} + \sum_{k \, = \, 0}^{\lfloor n/2 \rfloor-1} \{s_k[s_2]\cdot s_{n-2k}\}_{2n-2k-2\lfloor n/2 \rfloor-1\le\lambda_1\le 2n-4k-2}=\\
       & \sum_{d \, = \, 0}^{\lfloor n/2 \rfloor} \{s_d[s_2]\cdot s_{n-2d}\}_{\lambda_1\le 2(n-2d)} + \sum_{d \, = \, 0}^{\lfloor n/2 \rfloor-1} \{s_d[s_2]\cdot s_{n-2d}\}_{2n-2d-2\lfloor n/2 \rfloor-1\le\lambda_1\le 2(n-2d)-2}= \label{eqn:second-to-last}\\
       & \sum_{d \, = \, 0}^{\lfloor n/2 \rfloor} \{s_d[s_2]\cdot s_{n-2d}\}_{\lambda_1\le 2(n-2d)} + \sum_{d \, = \, 0}^{\lfloor n/2 \rfloor-1} \{s_d[s_2]\cdot s_{n-2d}\}_{\lambda_1\le 2(n-2d)-2} =\\
       & \sum_{d \, = \, 0}^{\lfloor n/2 \rfloor} \{s_d[s_2]\cdot s_{n-2d}\}_{\lambda_1\le 2(n-2d)} + \sum_{d \, = \, 0}^{\lfloor n/2 \rfloor} \{s_d[s_2]\cdot s_{n-2d}\}_{\lambda_1\le 2(n-2d)-2}
    \end{align}
    where the second-to-last equality follows from the fact that each $s_{\lambda}$ appearing in the Schur expansion of $s_d[s_2]\cdot s_{n-2d}$  satisfies  
    \begin{equation}\lambda_1 \ge n-2d = 2n-2d-(n-1)-1 \ge 2n-2d-2\lfloor n/2 \rfloor-1.\end{equation}
\end{proof}

\subsection{The ideal $I_{n,a}^\MMM$}
We want to understand the defining ideal $\gr \, \II(\MMM_{n,a})$ of $R(\MMM_{n,a})$. As in the case of $\gr \, \II(\MMM_n)$ (Lemma~\ref{lem:matching-ideal-containment}), some elements of $\gr \, \II(\MMM_{n,a})$ are not difficult to obtain.

\begin{defn}
    \label{def:conjugacy-ideal}
    Suppose $0 \leq a \leq n$ and $2 \mid (n-a)$. Let $I_{n,a}^\MMM \subseteq \CC[\xxx_{n \times n}]$ be the ideal 
    \begin{align}
    I_{n,a}^\MMM &= I_n^\MMM + (x_{1,1} + \cdots + x_{n,n}) + \left(  \prod_{i \, \in \, S} x_{i,i} \,:\, S \subseteq [n], \, |S| > a \right) \label{eqn:i-def-one}\\
    &= \gr \, \II(\MMM_n) + (x_{1,1} + \cdots + x_{n,n}) + \left(  \prod_{i \, \in \, S} x_{i,i} \,:\, S \subseteq [n], \, |S| > a \right). \label{eqn:i-def-two}
\end{align}
The equality \eqref{eqn:i-def-two} is justified by
Proposition~\ref{prop:matching-ideal-equality}.
\end{defn}

  In other words, the ideal $I_{n,a}^\MMM$ is generated by $I_n^\MMM$ together with the diagonal sum $x_{1,1} + \cdots + x_{n,n}$ and any product of $a+1$ variables on the diagonal.
  The analog of Lemma~\ref{lem:matching-ideal-containment} is as follows.

\begin{lemma}
    \label{lem:conjugacy-ideal-containment}
    We have $I_{n,a}^\MMM \subseteq \gr \, \II(\MMM_{n,a})$.
\end{lemma}

Unlike in the case of $\MMM_n$ or $\PM_n$, we do {\bf not} have $I_{n,a}^\MMM = \gr \, \II(\MMM_{n,a})$ in general. We leave finding an explicit generating set of $\gr \, \II(\MMM_{n,a})$ as an open problem.

\begin{proof}
    Since $\MMM_{n,a} \subseteq \MMM_n$, we need only show that the diagonal sum $x_{1,1} + \cdots + x_{n,n}$ and the product $\prod_{i \, \in \, S} x_{i,i}$ lie in $\gr \, \II(\MMM_{n,a})$ for $|S| > a$. Indeed, we have $x_{1,1} + \cdots + x_{n,n} - a \in \II(\MMM_{n,a})$ and therefore $x_{1,1} + \cdots + x_{n,n} \in \gr \, \II(\MMM_{n,a})$.
    Furthermore, if $|S| > a$ we have $\prod_{i \, \in \, S} x_{i,i} \in \II(\MMM_{n,a})$ and therefore $\prod_{i \, \in \, S} x_{i,i} \in \gr \, \II(\MMM_{n,a})$.
\end{proof}





We want to show that certain degrees of the graded $\symm_n$-module $R(\MMM_{n,a})$ are annihilated by the symmetrization operator 
$\eta_j = \sum_{w \in \symm_j} w$. This will be done in Lemma~\ref{lem:conjugacy-module-annihilation} below. In order to prove Lemma~\ref{lem:conjugacy-module-annihilation} we require two technical results; the first is as follows.

\begin{lemma}\label{lem:f-not-inj}
    Fix $A=\{i_1,j_1,\cdots,i_d,j_d\}$ consisting of $d$ pairs of integers in $[n]$ and suppose $B\subseteq [n]$ satisfies  $\lvert A\rvert >\lvert B \rvert +a$. For each function $f:A\rightarrow B$, at least one of the two following statements is true.
    \begin{itemize}
        \item There exist integers $1\le k_1 < k_2 < \cdots < k_{a+1} \le d$ such that $f(i_{k_\ell})=f(j_{k_\ell})$ for all $\ell\in[a+1]$.
        \item There exist integers $1 \le t_1 < t_2 \le d$ such that $f(\{i_{t_1},j_{t_1}\}) \cap f(\{i_{t_2},j_{t_2}\}) \neq \varnothing$.
    \end{itemize}
\end{lemma}
\begin{proof}
    We prove this lemma by contradiction. Assume that neither of the two statements holds. Since the first statement does not hold, we have   
    \begin{equation}
         d^{\prime} := |\{ 1 \leq k \leq d \,:\, f(i_k) = f(j_k) \}| \leq a.
    \end{equation}
    Define $C \subseteq A$ by
    \begin{equation}
        C := \{ i_k, j_k \,:\, f(i_k) = f(j_k) \}.
    \end{equation}
    Then $|C| = 2d^{\prime} \leq 2a$, $|f(C)| = d^{\prime}$, and the restriction $f \mid_{A \setminus C}$ of $f$ to $A \setminus C$ is injective. 
     Since the second statement does not hold, we have 
    \begin{equation} 
    \label{eqn:intersection-images-empty}
    f(A\setminus C) \cap f(C) = \varnothing.
    \end{equation}
    We compute 
    \begin{multline}
    \label{eqn:leading-to-contradiction}
        \lvert f(A) \rvert = \lvert f(A\setminus C) \rvert + \lvert f(C) \rvert = \lvert f(A\setminus C) \rvert + d^{\prime} \\ = \lvert A\setminus C \rvert + d^{\prime}   = \lvert A \rvert - d^{\prime} > \lvert B \rvert +a -d^{\prime} \ge \lvert B \rvert
    \end{multline}
    where the first equality is justified by \eqref{eqn:intersection-images-empty}, the second equality uses $|f(C)| = d^{\prime}$, the third equality holds because $f \mid_{A \setminus C}$ is injective, the fourth equality holds because $|C| = 2d^{\prime}$, the strict inequality $>$ holds because $|A| > |B| + a$, and the weak inequality $\geq$ uses $a \geq d^{\prime}$. But \eqref{eqn:leading-to-contradiction} implies $|f(A)|> |B|$, which is a contradiction.
\end{proof}
 
For $1 \leq p \leq n$ we have the symmetrization operator $\eta_p = \sum_{w \in \symm_p} w$ over $\symm_p \subseteq \symm_n$. We use Lemma~\ref{lem:f-not-inj} to show that the images of certain matching monomials under $\eta_p$ lie in the ideal $I_{n,a}^\MMM$.

\begin{lemma}
\label{lem:matching-monomial-membership}
    Suppose $2 \mid (n-a)$.
    Let $w = (i_1,j_1) (i_2,j_2) \cdots (i_d,j_d) \in \symm_n$ be a matching with $d$ matched pairs. For $1 \leq p \leq n$ with $p > n - 2d + a$ we have $\eta_p \cdot \mmm(w) \in I_{n,a}^\MMM$.
\end{lemma}

The proof of Lemma~\ref{lem:matching-monomial-membership} is involved, but worth it. It will enable us to fruitfully apply Lemma~\ref{lem:eta-annihilation} to the ring $R(\MMM_{n,a})$.

\begin{proof}
We prove this lemma by induction on $a$. If $a=0$ (and therefore $n$ is even), the desired statement follows from Lemma~\ref{lem:pm-symmetrizer-annihilation}. 
We assume that $a>0$ and that the lemma holds for all smaller $a$'s. 

Let $S:=\{i_1,j_1,\cdots,i_d,j_d\}$ be the indices in the 2-cycles of $w$ and $T:=S\cap [p]$. Note that
\begin{equation}\eta_p \cdot \mmm(w) \doteq \sum_{f:T\hookrightarrow [p]} \prod_{k=1}^d x_{f(i_k),f(j_k)} = \sum_{f:T\rightarrow [p]} \prod_{k=1}^d x_{f(i_k),f(j_k)} - \sum_{\substack{f:T\rightarrow [p] \\ \text{not injective}}} \prod_{k=1}^d x_{f(i_k),f(j_k)} 
\end{equation}
with the convention $f(\ell)=\ell$ if $\ell \not\in T$. We claim that both
\begin{equation}\label{eq:match-ann-one}
    \sum_{f:T\rightarrow [p]} \prod_{k=1}^d x_{f(i_k),f(j_k)} \in I_{n,a}^\MMM
\end{equation}
and 
\begin{equation}\label{eq:match-ann-two}
    \sum_{\substack{f:T\rightarrow [p] \\ \text{not injective}}} \prod_{k=1}^d x_{f(i_k),f(j_k)} \in I_{n,a}^\MMM.
\end{equation}

To prove the membership \eqref{eq:match-ann-one}, one first shows 
\begin{equation}\label{eq:match-ann-one-help}
\sum_{f:T\rightarrow [p]} \prod_{k=1}^d x_{f(i_k),f(j_k)} \doteq \sum_{f:T\rightarrow ([n]\setminus[p]) \setminus (S\setminus T)} \prod_{k=1}^d x_{f(i_k),f(j_k)} \mod{I_{n,a}^\MMM}.
\end{equation}
The proof of Equation~\eqref{eq:match-ann-one-help} is similar to that of Equation~\eqref{eq:pm-ann-help} and is left to the reader.
Since 
\begin{multline}
\lvert ([n]\setminus[p]) \setminus (S\setminus T) \rvert = n-p-(2d-\lvert T \rvert) \\ =  n-p-2d+\lvert T \rvert < n-(n-2d+a)-2d-\lvert T\rvert = \lvert T \rvert - a, 
\end{multline}
by Lemma~\ref{lem:f-not-inj} for each $f:T\rightarrow ([n]\setminus[p]) \setminus (S\setminus T)$ indexing the sum on the right hand side of \eqref{eq:match-ann-one-help}, we have $\prod_{k=1}^d x_{f(i_k),f(j_k)} \in I_{n,a}^\MMM$. Thus the right hand side of \eqref{eq:match-ann-one-help} lies in $I_{n,a}^\MMM$, which proves \eqref{eq:match-ann-one}.

We turn to the proof of the membership~\eqref{eq:match-ann-two}. Since each row product $x_{i,j}\cdot x_{i,j^{\prime}}$, each column product $x_{i,j}\cdot x_{i^{\prime},j}$, and each symmetric difference $x_{i,j}-x_{j,i}$ lie in $I_{n,a}^\MMM$, we have
\begin{quote}$\prod_{k=1}^d x_{f(i_k),f(j_k)} \in I_{n,a}^\MMM$
for all $f:T\rightarrow[p]$ such that $f(\{i_k,j_k\})\cap f(\{i_\ell,j_\ell\}) \neq \varnothing$ for some distinct $k,\ell\in[d]$. \end{quote} 

Therefore, we have 
\begin{equation} \sum_{\substack{f:T\rightarrow [p] \\ \text{not injective}}} \prod_{k=1}^d x_{f(i_k),f(j_k)} \equiv \sum_{m=1}^d A_m \mod{I_{n,a}^\MMM} \end{equation}
where $A_m$ is defined by
\begin{equation}A_m\coloneqq\sum_f \prod_{k=1}^d x_{f(i_k),f(j_k)}\end{equation} 
and the sum is over all functions $f:T\rightarrow[p]$ such that 
\begin{itemize}
    \item $f(\{i_k,j_k\})\cap f(\{i_\ell,j_\ell\}) = \varnothing$ for each pair of distinct integers $k,\ell\in[d]$, and
    \item there exist exactly $m$ integers $k\in[d]$ such that $f(i_k)=f(j_k)$.
\end{itemize}

In order to prove the membership~\eqref{eq:match-ann-two}, it suffices to show that 
\begin{equation}\label{eq:match-ann-three}
    A_m \in I_{n,a}^\MMM
\end{equation}
for all integers $m\in[d]$. If $m>a$, the membership~\eqref{eq:match-ann-three} holds since each product of $> a$ diagonal variables lies in $I_{n,a}^\MMM$, so assume $m \le a$.  Group the terms in the sum defining $A_m$ according to the positions $k$ where $f(i_k)=f(j_k)$ so that $A_m$ is decomposed into $\binom{d}{m}$ smaller sums. We claim that all these smaller sums lie in $I_{n,a}^\MMM$. Without loss of generality, it suffices to show that 
\begin{equation}\label{eq:match-ann-two-help}
    \sum_f \prod_{k=1}^d x_{f(i_k),f(j_k)} \in I_{n,a}^\MMM
\end{equation}
where the sum is over all functions $f:T\rightarrow [p]$ such that
\begin{itemize}
    \item $f(\{i_k,j_k\})\cap f(\{i_\ell,j_\ell\}) = \varnothing$ for each pair of distinct integers $k,\ell\in[d]$, and
    \item $f(i_k)=f(j_k)$ if and only if $k\in[m]$. (Note that this means $i_k,j_k \in T$.)
\end{itemize}
Let $S^{\prime}=S\setminus \{i_1,j_1,\cdots,i_m,j_m\}$ and $T^{\prime}=T\setminus \{i_1,j_1,\cdots,i_m,j_m\}$. The polynomial in \eqref{eq:match-ann-two-help} satisfies
\begin{align} 
\sum_f \prod_{k=1}^d x_{f(i_k),f(j_k)} &\equiv \Bigg(\sum_{t=1}^p x_{t,t}\Bigg)^m \cdot \Bigg( \sum_{f:T^{\prime}\hookrightarrow [p]} \prod_{k=m+1}^d x_{f(i_k),f(j_k)} \Bigg) \\
&\equiv \Bigg(-\sum_{t=p+1}^n x_{t,t}\Bigg)^m \cdot \Bigg( \sum_{f:T^{\prime}\hookrightarrow [p]} \prod_{k=m+1}^d x_{f(i_k),f(j_k)} \Bigg) \\
&\equiv (-1)^m \sum_{t_1\neq t_2\neq\cdots\neq t_m \in [n]\setminus[p]} \Bigg(\prod_{\ell=1}^m x_{t_\ell,t_\ell}\Bigg) \cdot \Bigg( \sum_{f:T^{\prime}\hookrightarrow [p]} \prod_{k=m+1}^d x_{f(i_k),f(j_k)} \Bigg) \mod{I_{n,a}^\MMM}
\end{align}
where the first $\equiv$ holds because each row product $x_{i,j}\cdot x_{i,j^{\prime}}$ and each column product $x_{i,j}\cdot x_{i^{\prime},j}$ lie in $I_{n,a}^\MMM$ and the second $\equiv$ holds because $\sum_{i=1}^n x_{i,i} \in I_{n,a}^\MMM$. 

By the last paragraph, in order to show the ideal membership~\eqref{eq:match-ann-two-help}, it suffices to show  
\begin{equation}\label{eq:match-ann-final}
\Bigg(\prod_{\ell=1}^m x_{t_\ell,t_\ell}\Bigg) \cdot \Bigg( \sum_{f:T^{\prime}\hookrightarrow [p]} \prod_{k=m+1}^d x_{f(i_k),f(j_k)} \Bigg) \in I_{n,a}^\MMM
\end{equation}
for $p+1 \le t_1<t_2<\cdots<t_m\le n$.  If there exists some $\ell\in [m]$ such that $t_\ell \in S^{\prime}$, we have $t_\ell \in S^{\prime}\setminus T^{\prime}$ since $t_\ell>p$. Thus there exists $k\in[d]\setminus[m]$ such that $t_\ell=i_k$ and $f(i_k)=i_k$ where by convention $f(\ell) = \ell$ if $\ell \notin T'$.  Then the membership~\eqref{eq:match-ann-final} follows from that each row product $x_{i,j}\cdot x_{i,j^{\prime}}$ and each column product $x_{i,j}\cdot x_{i^{\prime},j}$ lie in $I_{n,a}^\MMM$. Thus we can assume that $t_\ell\not\in S^{\prime}$ for all $\ell\in[m]$. Consider the variable matrix $\overline{\xxx}$ obtained by removing rows $t_1,\cdots,t_m$ and columns $t_1,\cdots,t_m$ from $\xxx_{n\times n}$. Let $\CC[\overline{\xxx}]$ be the polynomial ring in these variables. Set 
$$n' := n-m, \quad \quad a' := a-m, \quad \quad d' := d-m,$$and define an ideal $I_{n^{\prime},a^{\prime}}^{\MMM^{\prime}} \subseteq \CC[\overline{\xxx}]$ using the same generators as in Definition~\ref{def:conjugacy-ideal}, but with respect to the smaller matrix $\overline{\xxx}$. Then we have \begin{equation}\sum_{f:T^{\prime}\hookrightarrow [p]} \prod_{k \, = \,m+1}^d x_{f(i_k),f(j_k)} \in I_{n^{\prime},a^{\prime}}^{\MMM^{\prime}}\end{equation} by the inductive hypothesis, since \begin{equation}p>n-2d+a=(n^{\prime}+m)-2(d^{\prime}+m)+(a^{\prime}+m)=n^{\prime}-2d^{\prime}+a^{\prime}.\end{equation} (Here $a^{\prime}=a-m \ge 0$ since we have assumed that $m\le a$ below Equation~\eqref{eq:match-ann-three}.)  The membership~\eqref{eq:match-ann-final} follows from the easily verified containment 
    \begin{equation}
    \Bigg(\prod_{\ell=1}^m x_{t_\ell,t_\ell}\Bigg) \cdot  I^{\MMM^{\prime}}_{n^{\prime},a^{\prime}} \subseteq I^{\MMM}_{n,a}
    \end{equation}
    of subspaces of $\CC[\xxx_{n \times n}]$.
\end{proof}

As promised, Lemma~\ref{lem:matching-monomial-membership} has a consequence for annihilation under the $\eta$-operators. This imposes a crucial restriction on the irreducible representations which can appear in the graded pieces of $R(\MMM_{n,a})$.

\begin{lemma}
    \label{lem:conjugacy-module-annihilation}
    Suppose $2 \mid (n-a)$. For any $1 \leq j \leq n$ with $j > n - 2d + a$ we have $$\eta_j \cdot R(\MMM_{n,a})_d = 0.$$
    In particular, if an $\symm_n$-irreducible $V^\lambda$ appears in $R(\MMM_{n,a})_d$ we have $\lambda_1 \leq n - 2d + a$.
\end{lemma}

\begin{proof}
    By Theorem~\ref{thm:matching-monomial-basis} and the fact that $\MMM_{n,a} \subseteq \MMM_n$, the set $\{ \mmm(w) \,:\, w \in \MMM_n \}$ of matching monomials descends to a spanning set of $R(\MMM_{n,a})$. The first statement follows from Lemma~\ref{lem:matching-monomial-membership}. The last statement follows from Lemma~\ref{lem:eta-annihilation}.
\end{proof}

\subsection{Surjections of $\symm_n$-modules} The proof of Lemma~\ref{lem:conjugacy-module-annihilation} used the fact that the matching monomials $\{ \mmm(w) \,:\, w \in \MMM_n \}$ descend to a spanning set of $R(\MMM_{n,a})$. In fact, this spanning set can be trimmed down. This results in a degree bound on the module $R(\MMM_{n,a})$.

\begin{lemma}
    \label{lem:R-degree-bound}
    Suppose $2 \mid (n-a)$. The set of matching monomials
   $$ \bigsqcup_{b \, \geq \, a} \{ \mmm(w) \,:\, w \in \MMM_{n,b} \}$$
    corresponding to matchings with at least $a$ fixed points descends to a spanning set of $R(\MMM_{n,a})$. In particular, we have 
    $$ R(\MMM_{n,a})_d = 0 \quad \text{for $d > (n-a)/2$.}$$
\end{lemma}

\begin{proof}
    The second statement follows from the first. For the first statement, observe that if $w \in \MMM_{n,b}$ for $b > a$ then $\mmm(w)$ vanishes on the locus $\MMM_{n,a}$ so that $\mmm(w) \in \II(\MMM_{n,a})$ and therefore $\mmm(w) \in \gr \, \II(\MMM_{n,a})$.  Since $\{ \mmm(w) \,:\, w \in \MMM_n \}$ descends to a spanning set of $R(\MMM_{n,a})$ we are done.
\end{proof}

Lemma~\ref{lem:R-degree-bound} allows us to restrict our attention to degrees $d \leq (n-a)/2$. In order to understand the structure of these smaller degrees, we consider the quotients
$\CC[\xxx_{n \times n}]_{\leq d}/(\II(\MMM_{n,a}) \cap \CC[\xxx_{n \times n}]_{\leq d})$. These quotient spaces have the following spanning sets.

\begin{lemma}
    \label{lem:d-quotient-spanning}
    For any $0 \leq d \leq (n-a)/2$, the vector space $\CC[\xxx_{n \times n}]_{\leq d}/(\II(\MMM_{n,a}) \cap \CC[\xxx_{n \times n}]_{\leq d})$ is spanned by 
    $$\{ \mmm(w) \,:\, \text{$w \in \MMM_n$ has exactly $d$ matched pairs} \}.$$
\end{lemma}

\begin{proof}
    Lemmas~\ref{lem:basis-transfer} and \ref{lem:R-degree-bound} imply that $\{\mmm(w):\text{$w\in\MMM_n$ has at most $d$ matched pairs}\}$ descends to a spanning set of $\CC[\xxx_{n \times n}]_{\leq d}/(\II(\MMM_{n,a}) \cap \CC[\xxx_{n \times n}]_{\leq d})$. Thus it suffices to show that for any involution $w^{\prime}\in \MMM_n$ with $d^{\prime}<d$ matched pairs, $\mmm(w^{\prime})\in\spa(\{\mmm(w):\text{$w\in\MMM_n$ has $d$ matched pairs}\})$ modulo $\II(\MMM_{n,a})\cap\CC[\xxx_{n\times n}]_{\le d}$. By induction, it suffices to show the following claim.
      \begin{quote}
          {\bf Claim: }
            {\em If $w' \in \MMM_n$ has $d' < d$ matched pairs, we have $$\mmm(w^{\prime})\in\spa(\{\mmm(w):\text{$w\in\MMM_n$ has $d^{\prime}+1$ matched pairs}\})$$ inside the quotient space $\CC[\xxx_{n\times n}]_{\le d}/(\II(\MMM_{n,a})\cap\CC[\xxx_{n\times n}]_{\le d})$.}
      \end{quote}
      
      In fact, we have the following identification.
      \begin{align}\label{eq:identify-ungrad}
          \CC[\xxx_{n\times n}]/\II(\MMM_{n,a}) &\cong \CC[\MMM_{n,a}] \\
          \nonumber \mmm(w) &\mapsto \sum_{u}u
      \end{align}
      where the summation is over all involutions $u\in \MMM_{n,a}$ such that $w \prec u$. Here we write $w \prec u$ if $u$ is obtained by adding matched pairs to $w$. Further, we write $w \precdot u$ if $u$ is obtained by adding one matched pair to $w$.

      Under the identification \eqref{eq:identify-ungrad}, we have
      \[
          \sum_{w^{\prime}\precdot w^{\prime\prime}}\mmm(w^{\prime\prime}) = \sum_{w^{\prime}\precdot w^{\prime\prime}}\sum_{\substack{v^{\prime\prime} \in \MMM_{n,a} \\ w^{\prime\prime}\prec v^{\prime\prime}}}v^{\prime\prime} = 
          \Big(\frac{n-a}{2}-d^{\prime}\Big)\cdot \sum_{\substack{v^{\prime}\in\MMM_{n,a} \\ w^{\prime} \prec v^{\prime}}}v^{\prime} = \Big(\frac{n-a}{2}-d^{\prime}\Big) \cdot \mmm(w^{\prime})
      \]
      where the second $=$ is true since for each $v^{\prime}\in\MMM_{n,a}$ such that $w^{\prime}\prec v^{\prime}$, there are exactly $\frac{n-a}{2}-d^{\prime}$ involutions $w^{\prime\prime}$ such that $w^{\prime}\precdot w^{\prime\prime} \prec v^{\prime}$. Then
      \begin{equation}\label{eq:ungrad-map}
      \mmm(w^{\prime}) \equiv \frac{2}{n-2d^{\prime}-a} \cdot \sum_{w^{\prime}\precdot w^{\prime\prime}}\mmm(w^{\prime\prime}) \mod \II(\MMM_{n,a}) \cap \CC[\xxx_{n \times n}]_{\leq d}
      \end{equation}
      and the result follows.
\end{proof}

Lemma~\ref{lem:d-quotient-spanning} has the following representation-theoretic consequence relating the modules $R(\MMM_n)$ and $R(\MMM_{n,a})$.

\begin{lemma}
    \label{lem:surjection-lemma}
    For $0 \leq d \leq (n-a)/2$ there exists a surjection of ungraded $\symm_n$-modules
    $$\varphi_{n,a,d}: R(\MMM_n)_d \twoheadrightarrow R(\MMM_{n,a})_{\leq d}.$$
\end{lemma}

\begin{proof}
    Lemma~\ref{lem:refined-orbit-harmonics} gives an isomorphism of ungraded $\symm_n$-modules 
    \begin{equation}
    \label{eqn:orbit-harmonics-isomorphism-degree-restricted}
        R(\MMM_{n,a})_{\leq d} \cong \CC[\xxx_{n \times n}]_{\leq d}/(\II(\MMM_{n,a}) \cap \CC[\xxx_{n \times n}]_{\leq d}).
    \end{equation}
    This given, it suffices to describe a surjection of ungraded $\symm_n$-modules 
    \begin{equation}
        R(\MMM_n)_d \twoheadrightarrow \CC[\xxx_{n \times n}]_{\leq d}/(\II(\MMM_{n,a}) \cap \CC[\xxx_{n \times n}]_{\leq d}).
    \end{equation}
    Theorem~\ref{thm:matching-monomial-basis} implies that the set $\BBB$ of matching monomials $\mmm(w)$ corresponding to involutions $w \in \MMM_n$ with exactly $d$ matched pairs descends to a basis of $R(\MMM_n)$. Lemma~\ref{lem:d-quotient-spanning} implies that $\BBB$ descends to a spanning set of $\CC[\xxx_{n \times n}]_{\leq d}/(\II(\MMM_{n,a}) \cap \CC[\xxx_{n \times n}]_{\leq d})$. The map $\mmm(w) \mapsto \mmm(w)$ induces the desired surjection of $\symm_n$-modules.
\end{proof}

The quotient module $\CC[\xxx_{n \times n}]_{\leq d}/(\II(\MMM_{n,a}) \cap \CC[\xxx_{n \times n}]_{\leq d})$ appearing in the above proof has a combinatorial interpretation. Define a group algebra element $\epsilon_{n,a,d} \in \CC[\symm_n]$ as the sum
\begin{equation}
    \epsilon_{n,a,d} := \sum_{w} w
\end{equation}
where $w$ ranges over involutions $w \in \MMM_{n,a}$ with $a$ fixed points such that $w(i) = i+1$ for $i = 1, 3, \dots, 2d-1$. For example, we have
$$\epsilon_{7,1,2} = (1,2)  (3,4) \cdot ( (5,6) + (5,7) + (6,7)) \in \CC[\symm_7].$$
Observe that the polynomial $\mmm((1,2) (3,4)) = x_{1,2} \cdot x_{3,4}$ evaluates to 1 on the terms 
$$(1,2)(3,4)(5,6), (1,2)(3,4)(5,7),\text{ and }(1,2)(3,4)(6,7)$$ in the expansion of the above product, and that $\mmm((1,2) (3,4))$ vanishes on all other points of $\MMM_{7,1}$. These observations lead to the following lemma.

\begin{lemma}
    \label{lem:L-module-identification}
    Let $\symm_n$ act on $\CC[\symm_n]$ by conjugation.
    If $2 \mid (n-a)$ and $0 \leq d \leq (n-a)/2$, let $L_{n,a,d} \subseteq \CC[\symm_n]$ be the cyclic $\symm_n$-submodule geerated by $\epsilon_{n,a,d}$. In symbols, we have
    $$L_{n,a,d} := \mathrm{span} \{ w \cdot \epsilon_{n,a,d} \cdot w^{-1} \,:\, w \in \symm_n \}.$$
    There hold isomorphisms of ungraded $\symm_n$-modules
    $$R(\MMM_{n,a})_{\leq d} \cong  \CC[\xxx_{n \times n}]_{\leq d}/(\II(\MMM_{n,a}) \cap \CC[\xxx_{n \times n}]_{\leq d}) \cong L_{n,a,d}.$$
\end{lemma}

\begin{proof}
    Lemma~\ref{lem:d-quotient-spanning} says that $\CC[\xxx_{n \times n}]_{\leq d}/(\II(\MMM_{n,a}) \cap \CC[\xxx_{n \times n}]_{\leq d})$ is spanned by matching monomials $\mmm(w)$ for matchings $w \in \MMM_n$ with exactly $d$ matched pairs. Regarding $\mmm(w)$ as a polynomial function $\MMM_{n,a} \to \CC$, we see that 
    \begin{equation}
        \mmm(w): u \mapsto \begin{cases} 1  & \text{if $w(i) \neq i$ implies $w(i) = u(i)$,} \\ 0 & \text{otherwise.} \end{cases}
    \end{equation}
    Therefore the element $\sum_{u \in \MMM_{n,a}} \mmm(w)(u) \cdot u \in \CC[\MMM_{n,a}]$ may be expressed in terms of $\epsilon_{n,a,d}$ as $\sum_{u \in \MMM_{n,a}} \mmm(w)(u) \cdot u = w \cdot \epsilon_{n,a,d} \cdot w^{-1}$ and the result follows.
\end{proof}

Lemma~\ref{lem:L-module-identification} gives rise to natural containments of $\symm_n$-modules
\begin{equation}
\label{eqn:filtration}
    \CC = L_{n,a,0} \subset L_{n,a,1} \subset L_{n,a,2} \subset \cdots \subset L_{n,a,(n-a)/2} = \CC[\MMM_{n,a}].
\end{equation}
Furthermore, Theorem~\ref{thm:matching-monomial-basis} implies that 
\begin{equation}
\label{eqn:identification}
    \CC[\MMM_{n,n-2d}] \cong R(\MMM_n)_d
\end{equation}
as $\symm_n$-modules under the association $w \mapsto \mmm(w)$ for $w \in \MMM_{n,n-2d}$. This identification gives the $\symm_n$-module surjection of Lemma~\ref{lem:surjection-lemma} the following combinatorial intepretation. For $1 \leq i \leq n-1$,  write $s_i := (i,i+1) \in \symm_n$ for the adjancent transposition interchanging $i$ and $i+1$.

\begin{lemma}
    \label{lem:locus-surjection}
    Under the identifications of Lemma~\ref{lem:L-module-identification} and the isomorphism \eqref{eqn:identification}, the $\symm_n$-module surjection of Lemma~\ref{lem:surjection-lemma} becomes the $\symm_n$-homomorphism
    \begin{equation}
        \varphi_{n,a,d}: \CC[\MMM_{n,n-2d}] \twoheadrightarrow L_{n,a,d}
    \end{equation}
    characterized by
    \begin{equation}
        \varphi_{n,a,d}: s_1 s_3 \cdots s_{2d-1} \mapsto \epsilon_{n,a,d}.
    \end{equation}
\end{lemma}

\begin{proof}
    The matching monomial $\mmm(s_1 s_3 \cdots s_{2d-1}) = x_{1,2} x_{3,4} \cdots x_{2d-1,2d}$ evaluates to 1 on the summands of $\epsilon_{n,a,d}$ and evaluates to 0 on all other elements of $\MMM_{n,a,d}$.
\end{proof}

Since $\CC[\MMM_{n,n-2d}] = L_{n,n-2d,d}$, the filtration \eqref{eqn:filtration} allows us to consider the restriction of $\varphi_{n,a,d}$ to $L_{n,n-2d,k}$ for any $0 \leq k \leq d$. The image of this restriction is  as follows.

  \begin{lemma}\label{lem:induce-chain-surj}
      For all pairs of integers $0 \le k \le d\le (n-a)/2$, the $\symm_n$-module homomorphism $\varphi_{n,a,d}: \CC[\MMM_{n,n-2d}] \to L_{n,a,d}$ in Lemma~\ref{lem:locus-surjection} restricts to a surjective homomorphism between cyclic $\symm_n$-modules
      \begin{align*}
          L_{n,n-2d,k}&\twoheadrightarrow L_{n,a,k}\\
          \epsilon_{n,n-2d,k}&\mapsto c_k\epsilon_{n,a,k}
      \end{align*}
      for all $0\le k\le d$, where
      \[c_k\coloneqq\begin{cases}
          \prod_{i=k}^{d-1}\frac{n-2i-a}{n-2i-(n-2d)}, &\text{$k\le d-1$}\\
          1,&\text{$k=d$}.
      \end{cases}\]
  \end{lemma}

  \begin{proof}
      Since $c_k \neq 0$ and $\epsilon_{n,n-2d,k}, \epsilon_{n,a,k}$ are generators of the cyclic $\symm_n$-modules $L_{n,n-2d,k}, L_{n,a,k}$ (respectively), surjectivity will follow if we can prove $\varphi_{n,a,d}: \epsilon_{n,n-2d,k} \mapsto c_k \epsilon_{n,a,k}$. We prove this fact by downward induction on $k$.
    For $k=d$, since $L_{n,n-2d,k}=L_{n,n-2d,d}=\CC[\MMM_{n,n-2d}]$ and $\epsilon_{n,n-2d,d}=s_1s_3\cdots s_{2d-1}$, we have that $\epsilon_{n,n-2d,k}\mapsto c_k\epsilon_{n,a,k}$ by Lemma~\ref{lem:locus-surjection}. 
    
    Assume $\varphi_{n,a,d}: \epsilon_{n,n-2d,k}\mapsto c_k\epsilon_{n,a,k}$  for $k=t$ where $0<t\le d$. 
    We prove  $\varphi_{n,a,d}: \epsilon_{n,n-2d,t-1}\mapsto c_{t-1}\epsilon_{n,a,t-1}$. In fact, we have 
      \begin{equation}\epsilon_{n,n-2d,t-1} = \frac{2}{n-2(t-1)-(n-2d)}\cdot \epsilon_{n,n-2d,t}\end{equation}
      and
      \begin{equation}\epsilon_{n,a,t-1} = \frac{2}{n-2(t-1)-a}\cdot \epsilon_{n,a,t}\end{equation}
      by Equations~\eqref{eq:identify-ungrad} and \eqref{eq:ungrad-map}. By induction, we have
      \begin{align}
      \epsilon_{n,n-2d,t-1} \mapsto &\frac{2}{n-2(t-1)-(n-2d)} \cdot c_t\epsilon_{n,a,t} \\ &= \frac{2}{n-2(t-1)-(n-2d)} \cdot c_t \cdot \frac{n-2(t-1)-a}{2} \cdot \epsilon_{n,a,t-1} \\
      &= c_{t-1}\epsilon_{n,a,t-1}
      \end{align}
      and the proof is complete.
  \end{proof}

  We use the surjections of Lemma~\ref{lem:induce-chain-surj} in the special case $d = \frac{n-a}{2} - 1$. These surjections fit into a commutative diagram of ungraded $\symm_n$-modules as follows.

  \begin{lemma}\label{lem:a-fix-diagram}
      Suppose $2 \mid (n-a)$ and write $\delta := n \mod 2 \in \{0,1\}$. We have the following commutative diagram of ungraded $\symm_n$-modules where the vertical lines are northward containments and the horizontal lines are eastward surjections.
      \begin{scriptsize}
      \begin{displaymath}\displaystyle \scalebox{0.93}{
        \xymatrix{
        &&&&&R(\MMM_n)_{\frac{n-\delta}{2}}\\
        &&&&R_n({\MMM}_n)_{\frac{n-\delta-2}{2}} &R_n(\MMM_{n,\delta})_{\le\frac{n-\delta}{2}}\ar@{=}[u]\\
        &&&R(\MMM_n)_{\frac{n-\delta-4}{2}} &R(\MMM_{n,\delta+2})_{\le\frac{n-\delta-2}{2}}\ar@{=}[u]\ar@{->>}[r] &R(\MMM_{n,\delta})_{\le\frac{n-\delta-2}{2}}\ar@{-}[u]\\
        &&\begin{sideways}$\ddots$\end{sideways}\quad  & R(\MMM_{n,\delta+4})_{\le\frac{n-\delta-4}{2}}
        \ar@{=}[u]\ar@{->>}[r] &R(\MMM_{n,\delta+2})_{\le\frac{n-\delta-4}{2}}\ar@{-}[u]\ar@{->>}[r] &R(\MMM_{n,\delta})_{\le\frac{n-\delta-4}{2}}\ar@{-}[u]\\
        &R(\MMM_n)_1 &\vdots &\vdots\ar@{-}[u] &\vdots\ar@{-}[u] &\vdots\ar@{-}[u]\\
        R(\MMM_n)_0 &R(\MMM_{n,n-2})_{\le 1}\ar@{=}[u]\ar@{->>}[r] &\cdots\ar@{->>}[r] &R(\MMM_{n,\delta+4})_{\le 1}\ar@{-}[u]\ar@{->>}[r] &R(\MMM_{n,\delta+2})_{\le 1}\ar@{-}[u]\ar@{->>}[r] &R(\MMM_{n,\delta})_{\le 1}\ar@{-}[u]\\
        R(\MMM_{n,n})_{\le 0}\ar@{=}[u]\ar@{->>}[r] &R(\MMM_{n,n-2})_{\le 0}\ar@{-}[u]\ar@{->>}[r] &\cdots\ar@{->>}[r] &R(\MMM_{n,\delta+4})_{\le 0}\ar@{-}[u]\ar@{->>}[r] &R(\MMM_{n,\delta+2})_{\le 0}\ar@{-}[u]\ar@{->>}[r] &R(\MMM_{n,\delta})_{\le 0}\ar@{-}[u]\\
        \CC\ar@{=}[u] &\CC\ar@{=}[u] &\cdots &\CC\ar@{=}[u] &\CC\ar@{=}[u] &\CC\ar@{=}[u]
        }}
      \end{displaymath}
      \end{scriptsize}
  \end{lemma}

  \begin{proof}
      The vertical isomorphisms at the top of the diagram arise from the identification
      \begin{equation}
          R(\MMM_n)_d \cong \CC[\MMM_{n,n-2d}] \cong R(\MMM_{n,n-2d})_{\leq d}
      \end{equation}
      where the first isomorphism is \eqref{eqn:identification} and the second isomorphism follows from Lemma~\ref{lem:R-degree-bound}. The horizontal surjections are from the compositions
      \begin{equation}
          R(\MMM_{n,a})_{\leq d} \cong L_{n,a,d} \twoheadrightarrow L_{n,a-2,d} \cong R(\MMM_{n,a-2})_{\leq d}
      \end{equation}
      where the isomorphisms are those of Lemma~\ref{lem:L-module-identification} and the surjection is that of Lemma~\ref{lem:induce-chain-surj}.
  \end{proof}

The modules appearing in Lemma~\ref{lem:a-fix-diagram} are `bounded above' parts $R(\MMM_{n,a})_{\leq d}$ of the graded $\symm_n$-module $R(\MMM_{n,a})$. It will be more useful for us to have a corresponding result on the graded pieces $R(\MMM_{n,a})_d$ themselves. Fortunately, such a result is easily obtained from the commutativity of the diagram in Lemma~\ref{lem:a-fix-diagram}. The next proof uses the identification
\begin{equation}
   R(\MMM_{n,a})_d = R(\MMM_{n,a})_{\leq d}/R(\MMM_{n,a})_{\leq d-1}
\end{equation}
of $\symm_n$-modules.

\begin{lemma}
\label{lem:chain-of-surjections}
In the situation of Lemma~\ref{lem:a-fix-diagram}, we have chains of $\symm_n$-module surjections
    \label{lem:triangular-surjections}
     \begin{scriptsize}
      \begin{displaymath} \displaystyle \scalebox{1}{
        \xymatrix{
        &&&&&R_n(\MMM_{n,\delta})_{\frac{n-\delta}{2}}\\
        &&&&R(\MMM_{n,\delta+2})_{\frac{n-\delta-2}{2}}\ar@{->>}[r] &R(\MMM_{n,\delta})_{\frac{n-\delta-2}{2}}\\
        && & R(\MMM_{n,\delta+4})_{\frac{n-\delta-4}{2}}
        \ar@{->>}[r] &R(\MMM_{n,\delta+2})_{\frac{n-\delta-4}{2}}\ar@{->>}[r] &R(\MMM_{n,\delta})_{\frac{n-\delta-4}{2}}\\
        & &\begin{sideways}$\ddots$\end{sideways} &\vdots &\vdots &\vdots \\
         &R(\MMM_{n,n-2})_{1}\ar@{->>}[r] &\cdots\ar@{->>}[r] &R(\MMM_{n,\delta+4})_{1}\ar@{->>}[r] &R(\MMM_{n,\delta+2})_{ 1}\ar@{->>}[r] &R(\MMM_{n,\delta})_{1}\\
        R(\MMM_{n,n})_{0}\ar@{->>}[r] &R(\MMM_{n,n-2})_{0}\ar@{->>}[r] &\cdots\ar@{->>}[r] &R(\MMM_{n,\delta+4})_{0}\ar@{->>}[r] &R(\MMM_{n,\delta+2})_{0}\ar@{->>}[r] &R(\MMM_{n,\delta})_{0}
        }}
      \end{displaymath}
      \end{scriptsize}
\end{lemma}

\begin{proof}
    Let $G$ be a group, let $V_1, V_2$ be $\CC[G]$-modules, and let $W_i \subseteq V_i$ be a submodule for $i = 1,2$. Suppose there exists a surjective module homomorphism $\psi: V_1 \twoheadrightarrow V_2$ such that $\psi(W_1) \subseteq W_2$. Then the induced map $\overline{\psi}: V_1/W_1 \to V_2/W_2$ is a surjection. Applying this fact to Lemma~\ref{lem:a-fix-diagram} gives the result.
\end{proof}

\subsection{The graded character of $R(\MMM_{n,a})$} 
Our next lemmas describe the triangle of Lemma~\ref{lem:chain-of-surjections} in more detail. As an ungraded $\symm_n$-module, the direct sum over the column $R(\MMM_{n,a})_*$ of this triangle is simply $\CC[\MMM_{n,a}]$.

\begin{lemma}
    \label{lem:column-sums}
    If $a \equiv n \mod 2$, the direct sum of the modules in the column $R(\MMM_{n,a})_*$ of the triangle in Lemma~\ref{lem:chain-of-surjections} has ungraded Frobenius image
    \begin{equation}
        \sum_{d \, = \, 0}^{\frac{n-a}{2}} \Frob  \left( R(\MMM_{n,a})_d \right) = \Frob \left(  \bigoplus_{d \, = \, 0}^{\frac{n-a}{2}} R(\MMM_{n,a})_d \right) = \Frob( R(\MMM_{n,a} )) = s_{(n-a)/2}[s_2] \cdot s_a
    \end{equation}
\end{lemma}

\begin{proof}
    The first equality is clear. By Lemma~\ref{lem:R-degree-bound} we have $R(\MMM_{n,a})_d = 0$ for $d > (n-a)/2$, so the second equality holds. Orbit harmonics gives an isomorphism of ungraded $\symm_n$-modules
    $R(\MMM_{n,a}) \cong \CC[\MMM_{n,a}]$. As in the proof of Theorem~\ref{thm:matching-frobenius}, we have
    \begin{equation}
    \label{eqn:induction-product-identification}
        \Frob(\CC[\MMM_{n,a}]) = s_{(n-a)/2}[s_2] \cdot s_a,
    \end{equation}
    so the third equality holds.
\end{proof}

The proof of Lemma~\ref{lem:column-sums} only used the degree bound of Lemma~\ref{lem:R-degree-bound}, not the surjective maps of Lemma~\ref{lem:chain-of-surjections}. Our next lemma compares the Frobenius images of various entries of the triangle of Lemma~\ref{lem:chain-of-surjections}. Proving this lemma will use the surjectivity result.

\begin{lemma}
    \label{lem:surjection-chain-inequalities}
    Suppose $2 \mid (n-a)$ and $0 \leq d \leq (n-a)/2$. 
    \begin{enumerate}
    \item If $\frac{n-a}{2} \equiv d \mod 2$ we have the inequality of symmetric functions
    \begin{equation*}
        \Frob (R(\MMM_{n,a})_d) \leq \left\{ \Frob( R(\MMM_{n,a'})_d ) \right\}_{\lambda_1 \leq n - 2d + a}
    \end{equation*}
    where 
    $$a' := \frac{n+a}{2} - d.$$
    \item If $\frac{n-a}{2} \not\equiv d \mod 2$ we have the inequality of symmetric functions
    \begin{equation*}
        \Frob (R(\MMM_{n,a})_d) \leq \left\{ \Frob( R(\MMM_{n,a'+1})_d ) \right\}_{\lambda_1 \leq n - 2d + a}.
    \end{equation*}
    \end{enumerate}
\end{lemma}

\begin{proof}
    We start with the proof of (1). The parity assumption $(n-a)/2 \equiv d \mod 2$ means that $R(\MMM_{n,a'})_d$ is weakly west of (and in the same row as) $R(\MMM_{n,a})_d$ in the triangle of Lemma~\ref{lem:chain-of-surjections}. Lemma~\ref{lem:chain-of-surjections} applies to show that
    \begin{equation}
        \Frob (R(\MMM_{n,a})_d) \leq \Frob( R(\MMM_{n,a'})_d ).
    \end{equation}
    By Lemma~\ref{lem:conjugacy-module-annihilation} we have
    \begin{equation}
        \Frob( R(\MMM_{n,a'})_d ) = \{ \Frob( R(\MMM_{n,a'})_d ) \}_{\lambda_1 \leq n - 2d + a}.
    \end{equation}
    The desired inequality follows.

    For (2), the parity assumption $(n-a)/2 \not\equiv d \mod 2$ implies that $R(\MMM_{n,a'+1})_d$ is weakly west of (and in the same row as) $R(\MMM_{n,a})_d$. The desired inequality is derived in the same way as in (1).
\end{proof}

A picture should help in understanding Lemma~\ref{lem:surjection-chain-inequalities} and its proof. Replacing modules with dots, a southwest trapezoidal portion of the triangle in Lemma~\ref{lem:chain-of-surjections} looks as follows, where the horizontal surjections are suppressed.  
\begin{scriptsize}
\begin{center}
    \begin{tikzpicture}[scale = 0.2]

    \node (n,0) at (0,0) {$\circ$};

    \node (n-2,0) at (2,0) {$\circ$};
    \node (n-2,1) at (2,2) {$\circ$};

    \node (dots) at (6,0) {$\dots$};
    \node (dots) at (6,6) {\begin{sideways} $\ddots$ \end{sideways}};

    \node at (10,0) {$\circ$};
    \node at (10,2) {$\circ$};
    \node at (10,4) {$\circ$};
    \node at (10,6) {$\circ$};
    \node at (10,8.5) {$\vdots$};
    \node at (10,10){$\circ$};

    \node at (12,0) {$\circ$};
    \node at (12,2) {$\circ$};
    \node at (12,4) {$\star$};
    \node at (12,6) {$\circ$};
    \node at (12,8.5) {$\vdots$};
    \node at (12,10){$\circ$};
    \node at (12,12){$\blacktriangle$};

    \node at (15,0) {$\dots$};
    \node at (18,0) {$\circ$};
    \node at (18,2) {$\circ$};
    \node at (18,4) {$\circ$};
    \node at (18,6) {$\circ$};
    \node at (18,8.5) {$\vdots$};
    \node at (18,10){$\circ$};
    \node at (18,12){$\circ$};

    \node at (20,0) {$\circ$};
    \node at (20,2) {$\circ$};
    \node at (20,4) {$\bullet$};
    \node at (20,6) {$\circ$};
    \node at (20,8.5) {$\vdots$};
    \node at (20,10){$\circ$};
    \node at (20,12){$\circ$};

    \draw [-] (20,4) -- (12,12);
        
    \end{tikzpicture}
\end{center}
\end{scriptsize}
The symbol $\bullet$ represents the module $R(\MMM_{n,a})_d$. We draw a line segment of slope $-1$ starting at $\bullet$ and moving northwest. In the above picture, the endpoint $\blacktriangle$ of this line segment is a northwest corner of the big triangle; this means we are in Case 1 of Lemma~\ref{lem:surjection-chain-inequalities}. The $\star$ below the $\blacktriangle$ is in the same row as the $\bullet$; the corresponding module is $R(\MMM_{n,a'})_d$. 

It can happen that the northwest diagonal emanating from $\bullet = R(\MMM_{n,a})_d$ does not end on a northwest corner of the big triangle. An example of this is shown below, with the endpoint of the diagonal denoted with $\blacktriangle$.
\begin{scriptsize}
\begin{center}
    \begin{tikzpicture}[scale = 0.2]

    \node (n,0) at (0,0) {$\circ$};

    \node (n-2,0) at (2,0) {$\circ$};
    \node (n-2,1) at (2,2) {$\circ$};

    \node (dots) at (6,0) {$\dots$};
    \node (dots) at (6,6) {\begin{sideways} $\ddots$ \end{sideways}};

    \node at (10,0) {$\circ$};
    \node at (10,2) {$\star$};
    \node at (10,4) {$\circ$};
    \node at (10,6) {$\circ$};
    \node at (10,8.5) {$\vdots$};
    \node at (10,10){$\circ$};
    
    \node at (12,0) {$\circ$};
    \node at (12,2) {$\circ$};
    \node at (12,4) {$\circ$};
    \node at (12,6) {$\circ$};
    \node at (12,8.5) {$\vdots$};
    \node at (12,10){$\blacktriangle$};
    \node at (12,12){$\circ$};

    \node at (15,0) {$\dots$};
    \node at (18,0) {$\circ$};
    \node at (18,2) {$\circ$};
    \node at (18,4) {$\circ$};
    \node at (18,6) {$\circ$};
    \node at (18,8.5) {$\vdots$};
    \node at (18,10){$\circ$};
    \node at (18,12){$\circ$};

    \node at (20,0) {$\circ$};
    \node at (20,2) {$\bullet$};
    \node at (20,4) {$\circ$};
    \node at (20,6) {$\circ$};
    \node at (20,8.5) {$\vdots$};
    \node at (20,10){$\circ$};
    \node at (20,12){$\circ$};

    \draw [-] (20,2) -- (12,10);
        
    \end{tikzpicture}
\end{center}
\end{scriptsize}
The above picture corresponds to Case 2 of Lemma~\ref{lem:surjection-chain-inequalities}. This time, the $\star$ is one column west of the $\blacktriangle$, but still in the same row as the $\bullet$. The module corresponding to $\star$ is $R(\MMM_{n,a'+1})_d$. 

In either case, Lemma~\ref{lem:surjection-chain-inequalities} asserts the inequality of symmetric functions $$\Frob(\bullet) \leq \{ \Frob(\star) \}_{\lambda_1 \leq n-2d+a}.$$
The weaker inequality $\Frob(\bullet) \leq \Frob(\star)$ follows from the surjection $\star \twoheadrightarrow \bullet$ of Lemma~\ref{lem:chain-of-surjections}. Lemma~\ref{lem:conjugacy-module-annihilation} states that the terms $s_\lambda$ in the Schur expansion of  $\Frob(\star)$ satisfy $\lambda_1 \leq n-2a+d$, and the desired inequality $\Frob(\bullet) \leq \{ \Frob(\star) \}_{\lambda_1 \leq n-2d+a}$ follows.

Surprisingly, the inequalities of Lemma~\ref{lem:surjection-chain-inequalities} turn out to be equalities in all cases. Proving this fact involves  an interplay between the `column sum' equality of Lemma~\ref{lem:column-sums} and the inequalities of Lemma~\ref{lem:surjection-chain-inequalities}. The crucial idea is to partition the triangle of Lemma~\ref{lem:chain-of-surjections} according to its northwest-to-southeast diagonals. Lemma~\ref{lem:surjection-chain-inequalities} compares these diagonals with columns, and columns are understood by Lemma~\ref{lem:column-sums}.

\begin{lemma}
    \label{lem:surjection-chain-equalities}
    Suppose $a \equiv n \mod 2$ and $0 \leq d \leq (n-a)/2$. 
    \begin{enumerate}
    \item If $(n-a)/2 \equiv d \mod 2$ we have the equality of symmetric functions
    \begin{equation*}
        \Frob (R(\MMM_{n,a})_d) = \left\{ \Frob( R(\MMM_{n,a'})_d ) \right\}_{\lambda_1 \leq n - 2d + a}
    \end{equation*}
    where 
    $$a' := \frac{n+a}{2} - d.$$
    \item If $(n-a)/2 \not\equiv d \mod 2$ we have the equality of symmetric functions
    \begin{equation*}
        \Frob (R(\MMM_{n,a})_d) = \left\{ \Frob( R(\MMM_{n,a'+1})_d ) \right\}_{\lambda_1 \leq n - 2d + a}.
    \end{equation*}
    \end{enumerate}
\end{lemma}

\begin{proof}
    For any $a \equiv n \mod 2$ we define two graded $\symm_d$-modules $P_a, Q_a$ by
    \begin{align}
        P_a &:= R(\MMM_{n,a})_{\frac{n-a}{2}} \oplus R(\MMM_{n,a-2})_{\frac{n-a}{2} - 1} \oplus R(\MMM_{n,a-4})_{\frac{n-a}{2} - 2} \oplus \cdots \\
        Q_a &:=R(\MMM_{n,a-2})_{\frac{n-a}{2}} \oplus R(\MMM_{n,a-4})_{\frac{n-a}{2} - 1} \oplus  R(\MMM_{n,a-6})_{\frac{n-a}{2} - 2} \oplus \cdots 
    \end{align}
    where the summands form complete northwest-to-southeast diagonals in the triangle of Lemma~\ref{lem:chain-of-surjections}. The cases of $n = 10$ and $n = 11$, these modules look as follows, with the $P$-modules in red and the $Q$-modules in blue.

    \begin{center}
    \begin{scriptsize}
        \begin{tikzpicture}[scale = 0.4]
                \draw [fill=red!50] (20,0.5) -- (20.5,0) -- (20,-0.5) -- (19.5,0) -- (20,0.5);
                \draw [fill=blue!50] (22,0.5) -- (22.5,0) -- (22,-0.5) -- (21.5,0) -- (22,0.5);
                \draw [fill=red!50] (21.5,2) -- (22,2.5) -- (24.5,0) -- (24,-0.5) -- (21.5,2);
                \draw [fill=blue!50] (23.5,2) -- (24,2.5) -- (26.5,0) -- (26,-0.5) -- (23.5,2);
                \draw [fill=red!50] (23.5,4) -- (24,4.5) -- (28.5,0) -- (28,-0.5) -- (23.5,4);
                \draw [fill=blue!50] (25.5,4) -- (26,4.5) -- (30.5,0) -- (30,-0.5) -- (25.5,4);
                \draw [fill=red!50] (25.5,6) -- (26,6.5) -- (30.5,2) -- (30,1.5) -- (25.5,6);
                \draw [fill=blue!50] (27.5,6) -- (28,6.5) -- (30.5,4) -- (30,3.5) -- (27.5,6);
                \draw [fill=red!50] (27.5,8) -- (28,8.5) -- (30.5,6) -- (30,5.5) -- (27.5,8);
                \draw [fill=blue!50] (29.5,8) -- (30,8.5) -- (30.5,8) -- (30,7.5) -- (29.5,8);
                \draw [fill=red!50] (29.5,10) -- (30,10.5) -- (30.5,10) -- (30,9.5) -- (29.5,10);
                
                \node at (20,0) {$\circ$};

                \node at (22,0) {$\circ$};
                \node at (22,2) {$\circ$};

                \node at (24,0) {$\circ$};
                \node at (24,2) {$\circ$};
                \node at (24,4) {$\circ$};

                \node at (26,0) {$\circ$};
                \node at (26,2) {$\circ$};
                \node at (26,4) {$\circ$};
                \node at (26,6) {$\circ$};

                \node at (28,0) {$\circ$};
                \node at (28,2) {$\circ$};
                \node at (28,4) {$\circ$};
                \node at (28,6) {$\circ$};
                \node at (28,8) {$\circ$};

                \node at (30,0) {$\circ$};
                \node at (30,2) {$\circ$};
                \node at (30,4) {$\circ$};
                \node at (30,6) {$\circ$};
                \node at (30,8) {$\circ$};
                \node at (30,10) {$\circ$};

                \node at (20,-1.5) {{\color{red} $P_{11}$}};
                \node at (22,-1.5) {{\color{blue} $Q_{11}$}}; 
                \node at (24,-1.5) {{\color{red} $P_9$}};
                \node at (26,-1.5) {{\color{blue} $Q_9$}};
                \node at (28,-1.5) {{\color{red} $P_7$}};
                \node at (31,-1) {{\color{blue} $Q_7$}};
                \node at (31,1.5) {{\color{red} $P_6$}};
                \node at (31,3.5) {{\color{blue} $Q_6$}};
                \node at (31,5.5) {{\color{red} $P_3$}};
                \node at (31,7.5) {{\color{blue} $Q_3$}};
                \node at (31,9.5) {{\color{red} $P_1$}};

                \draw [fill=red!50] (0,0.5) -- (0.5,0) -- (0,-0.5) -- (-0.5,0) -- (0,0.5);
                \draw [fill=blue!50] (2,0.5) -- (2.5,0) -- (2,-0.5) -- (1.5,0) -- (2,0.5);
                \draw [fill=red!50] (1.5,2) -- (2,2.5) -- (4.5,0) -- (4,-0.5) -- (1.5,2);
                \draw [fill=blue!50] (3.5,2) -- (4,2.5) -- (6.5,0) -- (6,-0.5) -- (3.5,2);
                \draw [fill=red!50] (3.5,4) -- (4,4.5) -- (8.5,0) -- (8,-0.5) -- (3.5,4);
                \draw [fill=blue!50] (5.5,4) -- (6,4.5) -- (10.5,0) -- (10,-0.5) -- (5.5,4);
                \draw [fill=red!50] (5.5,6) -- (6,6.5) -- (10.5,2) -- (10,1.5) -- (5.5,6);
                \draw [fill=blue!50] (7.5,6) -- (8,6.5) -- (10.5,4) -- (10,3.5) -- (7.5,6);
                \draw [fill=red!50] (7.5,8) -- (8,8.5) -- (10.5,6) -- (10,5.5) -- (7.5,8);
                \draw [fill=blue!50] (9.5,8) -- (10,8.5) -- (10.5,8) -- (10,7.5) -- (9.5,8);
                \draw [fill=red!50] (9.5,10) -- (10,10.5) -- (10.5,10) -- (10,9.5) -- (9.5,10);
                
                \node at (0,0) {$\circ$};

                \node at (2,0) {$\circ$};
                \node at (2,2) {$\circ$};

                \node at (4,0) {$\circ$};
                \node at (4,2) {$\circ$};
                \node at (4,4) {$\circ$};

                \node at (6,0) {$\circ$};
                \node at (6,2) {$\circ$};
                \node at (6,4) {$\circ$};
                \node at (6,6) {$\circ$};

                \node at (8,0) {$\circ$};
                \node at (8,2) {$\circ$};
                \node at (8,4) {$\circ$};
                \node at (8,6) {$\circ$};
                \node at (8,8) {$\circ$};

                \node at (10,0) {$\circ$};
                \node at (10,2) {$\circ$};
                \node at (10,4) {$\circ$};
                \node at (10,6) {$\circ$};
                \node at (10,8) {$\circ$};
                \node at (10,10) {$\circ$};

                \node at (0,-1.5) {{\color{red} $P_{10}$}};
                \node at (2,-1.5) {{\color{blue} $Q_{10}$}}; 
                \node at (4,-1.5) {{\color{red} $P_8$}};
                \node at (6,-1.5) {{\color{blue} $Q_8$}};
                \node at (8,-1.5) {{\color{red} $P_6$}};
                \node at (11,-1) {{\color{blue} $Q_6$}};
                \node at (11,1.5) {{\color{red} $P_4$}};
                \node at (11,3.5) {{\color{blue} $Q_4$}};
                \node at (11,5.5) {{\color{red} $P_2$}};
                \node at (11,7.5) {{\color{blue} $Q_2$}};
                \node at (11,9.5) {{\color{red} $P_0$}};

                \node at (6,-3) {\begin{small}$n = 10$\end{small}};
                \node at (26,-3) {\begin{small}$n = 11$\end{small}};
        \end{tikzpicture}
    \end{scriptsize} 
    \end{center}
    For any $n$, the $P$-modules contain a summand on the diagonal of the triangle while the $Q$-modules do not. If $\delta = n \mod 2 \in \{0,1\}$ we have $Q_\delta = 0$.

    The summands of $P_a, Q_a$ for $a \equiv n \mod 2$ form a partition of the triangle of Lemma~\ref{lem:chain-of-surjections}. Another partition of this triangle separates it into columns. Lemma~\ref{lem:column-sums} implies
    \begin{equation}
    \label{eqn:PQ-sum}
        \sum_a \grFrob(P_a;1) + \sum_a \grFrob(Q_a;1) = \sum_a s_{(n-a)/2}[s_2] \cdot s_{a}
    \end{equation}
    where the sums range over all $0 \leq a \leq n$ with $a \equiv n \mod 2$.

    Lemma~\ref{lem:surjection-chain-inequalities} gives upper bounds on the graded $\symm_n$-modules $P_a$ and
    $Q_a$:
    \begin{align}
        \grFrob(P_a;q) &\leq \{ \grFrob(R(\MMM_{n,a});q) \}_{\lambda_1 \leq 2a}, \label{eqn:q-P-inequality} \\
        \grFrob(Q_a;q) &\leq \{ \grFrob(R(\MMM_{n,a});q) \}_{\lambda_1 \leq 2a-2}. \label{eqn:q-Q-inequality}
    \end{align}
    These upper bounds may  be visualized as follows. The inequality \eqref{eqn:q-P-inequality} compares a (red) $P$-module with the column whose top coincides with the northwest end of that module. Similarly, the inequality \eqref{eqn:q-Q-inequality} compares a (blue) $Q$-module with the column whose top is one unit west of the northwest end of that module. The case of the $P$-modules is shown on the left below while the case of the $Q$-modules is shown on the right. In the top two examples, the graded module $R(\MMM_{n,a})$ has the same lowest degree (i.e. degree 0) as the $P/Q$-modules. In the bottom two examples, the $P/Q$-modules vansh in degree 0, so the $R(\MMM_{n,a})$-module has smaller nonvanishing degrees.
    \begin{center}
    \begin{scriptsize}
        \begin{tikzpicture}[scale = 0.4]
        \draw [fill=gray!50] (23.65,4.35) -- (24.35,4.35) -- (24.35,-0.35) -- (23.65,-0.35) -- (23.65,4.35);
                \draw [fill=blue!50] (25.5,4) -- (26,4.5) -- (30.5,0) -- (30,-0.5) -- (25.5,4);

                \node at (20,0) {$\circ$};

                \node at (22,0) {$\circ$};
                \node at (22,2) {$\circ$};

                \node at (24,0) {$\circ$};
                \node at (24,2) {$\circ$};
                \node at (24,4) {$\circ$};

                \node at (26,0) {$\circ$};
                \node at (26,2) {$\circ$};
                \node at (26,4) {$\circ$};
                \node at (26,6) {$\circ$};

                \node at (28,0) {$\circ$};
                \node at (28,2) {$\circ$};
                \node at (28,4) {$\circ$};
                \node at (28,6) {$\circ$};
                \node at (28,8) {$\circ$};

                \node at (30,0) {$\circ$};
                \node at (30,2) {$\circ$};
                \node at (30,4) {$\circ$};
                \node at (30,6) {$\circ$};
                \node at (30,8) {$\circ$};
                \node at (30,10) {$\circ$};

                \draw [fill=gray!50] (3.65,4.35) -- (4.35,4.35) -- (4.35,-0.35) -- (3.65,-0.35) -- (3.65,4.35);
                \draw [fill=red!50] (3.5,4) -- (4,4.5) -- (8.5,0) -- (8,-0.5) -- (3.5,4);

                \node at (0,0) {$\circ$};

                \node at (2,0) {$\circ$};
                \node at (2,2) {$\circ$};

                \node at (4,0) {$\circ$};
                \node at (4,2) {$\circ$};
                \node at (4,4) {$\circ$};

                \node at (6,0) {$\circ$};
                \node at (6,2) {$\circ$};
                \node at (6,4) {$\circ$};
                \node at (6,6) {$\circ$};

                \node at (8,0) {$\circ$};
                \node at (8,2) {$\circ$};
                \node at (8,4) {$\circ$};
                \node at (8,6) {$\circ$};
                \node at (8,8) {$\circ$};

                \node at (10,0) {$\circ$};
                \node at (10,2) {$\circ$};
                \node at (10,4) {$\circ$};
                \node at (10,6) {$\circ$};
                \node at (10,8) {$\circ$};
                \node at (10,10) {$\circ$};

                \node at (0,-3) {};
        \end{tikzpicture}
        \begin{tikzpicture} [scale = 0.4]
                \draw [fill=blue!50] (27.5,6) -- (28,6.5) -- (30.5,4) -- (30,3.5) -- (27.5,6);
                \draw [fill=gray!50] (25.65,6.35) -- (26.35,6.35) -- (26.35,-0.35) -- (25.65,-0.35) -- (25.65,6.35);
                
                \node at (20,0) {$\circ$};

                \node at (22,0) {$\circ$};
                \node at (22,2) {$\circ$};

                \node at (24,0) {$\circ$};
                \node at (24,2) {$\circ$};
                \node at (24,4) {$\circ$};

                \node at (26,0) {$\circ$};
                \node at (26,2) {$\circ$};
                \node at (26,4) {$\circ$};
                \node at (26,6) {$\circ$};

                \node at (28,0) {$\circ$};
                \node at (28,2) {$\circ$};
                \node at (28,4) {$\circ$};
                \node at (28,6) {$\circ$};
                \node at (28,8) {$\circ$};

                \node at (30,0) {$\circ$};
                \node at (30,2) {$\circ$};
                \node at (30,4) {$\circ$};
                \node at (30,6) {$\circ$};
                \node at (30,8) {$\circ$};
                \node at (30,10) {$\circ$};

                \draw [fill=gray!50] (5.65,6.35) -- (6.35,6.35) -- (6.35,-0.35) -- (5.65,-0.35) -- (5.65,6.35);
                \draw [fill=red!50] (5.5,6) -- (6,6.5) -- (10.5,2) -- (10,1.5) -- (5.5,6);

                \node at (0,0) {$\circ$};

                \node at (2,0) {$\circ$};
                \node at (2,2) {$\circ$};

                \node at (4,0) {$\circ$};
                \node at (4,2) {$\circ$};
                \node at (4,4) {$\circ$};

                \node at (6,0) {$\circ$};
                \node at (6,2) {$\circ$};
                \node at (6,4) {$\circ$};
                \node at (6,6) {$\circ$};

                \node at (8,0) {$\circ$};
                \node at (8,2) {$\circ$};
                \node at (8,4) {$\circ$};
                \node at (8,6) {$\circ$};
                \node at (8,8) {$\circ$};

                \node at (10,0) {$\circ$};
                \node at (10,2) {$\circ$};
                \node at (10,4) {$\circ$};
                \node at (10,6) {$\circ$};
                \node at (10,8) {$\circ$};
                \node at (10,10) {$\circ$};
        \end{tikzpicture}
    \end{scriptsize} 
    \end{center}
    Setting $q = 1$ in \eqref{eqn:q-P-inequality}, \eqref{eqn:q-Q-inequality} and applying Lemma~\ref{lem:column-sums}, we see that
    \begin{align}
        \grFrob(P_a;1) &\leq \{ \grFrob(R(\MMM_{n,a});1) \}_{\lambda_1 \leq 2a} = \{ s_{(n-a)/2}[s_2] \cdot s_{a} \}_{\lambda_1 \leq 2a}, \label{eqn:P-bound} \\
        \grFrob(Q_a;1) &\leq \{ \grFrob(R(\MMM_{n,a});1) \}_{\lambda_1 \leq 2a-2} =  \{ s_{(n-a)/2}[s_2] \cdot s_{a} \}_{\lambda_1 \leq 2a-2}. \label{eqn:Q-bound}
    \end{align}
    Lemma~\ref{lem:s-identity-two} says that 
    \begin{equation}
    \label{eqn:s-sum-over-a}
        \sum_a \{ s_{(n-a)/2}[s_2] \cdot s_{a} \}_{\lambda_1 \leq 2a} + \sum_a \{ s_{(n-a)/2}[s_2] \cdot s_{a} \}_{\lambda_1 \leq 2a-2} = \sum_a s_{(n-a)/2}[s_2] \cdot s_{a}.
    \end{equation}
    Combining Equation~\eqref{eqn:PQ-sum} with the inequalities \eqref{eqn:P-bound}, \eqref{eqn:Q-bound}, we get the ungraded Frobenius images
    \begin{align}
        \grFrob(P_a;1) &= \{ s_{(n-a)/2}[s_2] \cdot s_{a} \}_{\lambda_1 \leq 2a} = \{ \grFrob(R(\MMM_{n,a});1) \}_{\lambda_1 \leq 2a}, \label{eqn:P-equality} \\
        \grFrob(Q_a;1) & = \{ s_{(n-a)/2}[s_2] \cdot s_{a} \}_{\lambda_1 \leq 2a-2} = \{ \grFrob(R(\MMM_{n,a});1) \}_{\lambda_1 \leq 2a-2}. \label{eqn:Q-equality}
    \end{align}
    Combining Equations~\eqref{eqn:P-equality}, \eqref{eqn:Q-equality} with the inequalities \eqref{eqn:q-P-inequality}, \eqref{eqn:q-Q-inequality} yields the equalities
    \begin{align}
        \grFrob(P_a;q) &= \{ \grFrob(R(\MMM_{n,a});q) \}_{\lambda_1 \leq 2a}, \label{eqn:q-P-equality} \\
        \grFrob(Q_a;q) &= \{ \grFrob(R(\MMM_{n,a});q) \}_{\lambda_1 \leq 2a-2}. \label{eqn:q-Q-equality}
    \end{align}
    Statement (1) of the lemma follows from \eqref{eqn:q-P-equality} and statement (2) follows from Equation~\eqref{eqn:q-Q-equality}.
\end{proof}

Lemma~\ref{lem:surjection-chain-equalities} is a powerful result which asserts isomorphisms between truncated versions of modules appearing in the triangle of Lemma~\ref{lem:chain-of-surjections}. Informally, it may be stated as follows.
\begin{quote}
    {\em Suppose $R(\MMM_{n,a})_d = \bullet$ is a module in the triangle of Lemma~\ref{lem:chain-of-surjections}. Find the modules $\blacktriangle$ and $\star$ as in the discussion before Lemma~\ref{lem:surjection-chain-equalities}. Then $\star$ and $\bullet$ are in the same row, and the Schur expansion of $\Frob(\bullet)$ may be obtained from $\Frob(\star)$ by removing any terms $s_\lambda$ for which $\lambda_1 > n - 2d + a$.}
\end{quote}
In terms of Lemma~\ref{lem:chain-of-surjections}, the bound $\lambda_1 \leq n-2d+a$ depends only on the northwest-to-southeast diagonal of the triangle containing $R(\MMM_{n,a})_d$, and weakens as one moves from east to west. The following consequence of Lemma~\ref{lem:surjection-chain-equalities} is what we will need to calculate $\grFrob(R(\MMM_{n,a});q)$.

\begin{lemma}
    \label{lem:module-transfer}
    Suppose $a \equiv n \mod 2$ and $0 \leq d \leq (n-a)/2$. We have the equality of symmetric functions
    \begin{equation}
        \Frob (R(\MMM_{n,a})_d) = \{ \Frob ( R(\MMM_{n,n-2d})_d ) \}_{\lambda_1 \leq n-2d+a}.
    \end{equation}
\end{lemma}

\begin{proof}
    If $d = (n-a)/2$, i.e. if $R(\MMM_{n,a})_d$ lies on the main diagonal of the triangle in Lemma~\ref{lem:chain-of-surjections}, we are done by Lemma~\ref{lem:conjugacy-module-annihilation}. Otherwise,
    consider iterating Lemma~\ref{lem:surjection-chain-equalities} in the form of the italicized quote above. 
    
    Starting with the module $\bullet := R(\MMM_{n,a})_d$, we obtain a module $\star =: \star_1$ in the same row as $R(\MMM_{n,a})_d$. The module $\star_1$ is strictly closer to the main diagonal of the triangle and there holds the symmetric function equality
    $$ \Frob(R(\MMM_{n,a})_d) = \{ \Frob(\star_1) \}_{\lambda_1 \leq n-2d+a}.$$
    If $\star_1$ lies on the main diagonal of the triangle in Lemma~\ref{lem:chain-of-surjections} we are done. Otherwise, we apply Lemma~\ref{lem:chain-of-surjections} to $\star_1$ to get a new module $\star_2$ which is strictly closer to the main diagonal than $\star_1$ and lies in the same row as $R(\MMM_{n,a})_d$. We have
     $$ \Frob(R(\MMM_{n,a})_d) = \{ \Frob(\star_1) \}_{\lambda_1 \leq n-2d+a} = \{ \Frob(\star_2) \}_{\lambda_1 \leq n-2d+a},$$
     where the second equality is justified by Lemma~\ref{lem:surjection-chain-equalities} since $\star_2$ lies on a northwest-to-southeast diagonal to the west of that containing $\star_1$. If $\star_2$ lies on the main diagonal of the triangle we are done. Otherwise, we form $\star_3$ from $\star_2$ using Lemma~\ref{lem:surjection-chain-equalities}, and so on.
\end{proof}

One possible sequence $\bullet, \star_1, \star_2, \dots $ as in the proof of Lemma~\ref{lem:module-transfer} is depicted below; this sequence terminates with $\star_3$ on the main diagonal. 
\begin{center}
\begin{scriptsize}
    \begin{tikzpicture}[scale = 0.3]
            \draw[-] (10,5) -- (8,7);
            \draw[-]  (7,5) -- (6,6);
            
            \node at (0,0) {$\circ$};
            
            \node at (1,0) {$\circ$};
            \node at (1,1) {$\circ$};

            \node at (2,0) {$\circ$};
            \node at (2,1) {$\circ$};
            \node at (2,2) {$\circ$};

            \node at (3,0) {$\circ$};
            \node at (3,1) {$\circ$};
            \node at (3,2) {$\circ$};
            \node at (3,3) {$\circ$};

            \node at (4,0) {$\circ$};
            \node at (4,1) {$\circ$};
            \node at (4,2) {$\circ$};
            \node at (4,3) {$\circ$};
            \node at (4,4) {$\circ$};

            \node at (5,0) {$\circ$};
            \node at (5,1) {$\circ$};
            \node at (5,2) {$\circ$};
            \node at (5,3) {$\circ$};
            \node at (5,4) {$\circ$};
            \node at (5.25,5) {$\star_3$};

            \node at (6,0) {$\circ$};
            \node at (6,1) {$\circ$};
            \node at (6,2) {$\circ$};
            \node at (6,3) {$\circ$};
            \node at (6,4) {$\circ$};
            \node at (6.25,5) {$\star_2$};
            \node at (6,6) {$\circ$};

            \node at (7,0) {$\circ$};
            \node at (7,1) {$\circ$};
            \node at (7,2) {$\circ$};
            \node at (7,3) {$\circ$};
            \node at (7,4) {$\circ$};
            \node at (7.25,5) {$\star_1$};
            \node at (7,6) {$\circ$};
            \node at (7,7) {$\circ$};

            \node at (8,0) {$\circ$};
            \node at (8,1) {$\circ$};
            \node at (8,2) {$\circ$};
            \node at (8,3) {$\circ$};
            \node at (8,4) {$\circ$};
            \node at (8,5) {$\circ$};
            \node at (8,6) {$\circ$};
            \node at (8,7) {$\circ$};
            \node at (8,8) {$\circ$};

            \node at (9,0) {$\circ$};
            \node at (9,1) {$\circ$};
            \node at (9,2) {$\circ$};
            \node at (9,3) {$\circ$};
            \node at (9,4) {$\circ$};
            \node at (9,5) {$\circ$};
            \node at (9,6) {$\circ$};
            \node at (9,7) {$\circ$};
            \node at (9,8) {$\circ$};
            \node at (9,9) {$\circ$};

            \node at (10,0) {$\circ$};
            \node at (10,1) {$\circ$};
            \node at (10,2) {$\circ$};
            \node at (10,3) {$\circ$};
            \node at (10,4) {$\circ$};
            \node at (10,5) {$\bullet$};
            \node at (10,6) {$\circ$};
            \node at (10,7) {$\circ$};
            \node at (10,8) {$\circ$};
            \node at (10,9) {$\circ$};
            \node at (10,10) {$\circ$};

            \node at (11,0) {$\circ$};
            \node at (11,1) {$\circ$};
            \node at (11,2) {$\circ$};
            \node at (11,3) {$\circ$};
            \node at (11,4) {$\circ$};
            \node at (11,5) {$\circ$};
            \node at (11,6) {$\circ$};
            \node at (11,7) {$\circ$};
            \node at (11,8) {$\circ$};
            \node at (11,9) {$\circ$};
            \node at (11,10) {$\circ$};
            \node at (11,11) {$\circ$};
    \end{tikzpicture}
\end{scriptsize}
\end{center}

We have all of the necessary tools to calculate the graded $\symm_n$-module structure of $R(\MMM_{n,a})$. The representation-theoretic Lemma~\ref{lem:module-transfer} plays a key role in the proof, as does the symmetric function identity in Lemma~\ref{lem:s-identity-one}.

\begin{theorem}
    \label{thm:conjugacy-module-character}
    Suppose $a \equiv n \mod 2$. The graded Frobenius image of $R(\MMM_{n,a})$ is given by
    \begin{equation}
        \grFrob(R(\MMM_{n,a});q) = \sum_{d \, = \, 0}^{(n-a)/2} \{
            s_d[s_2] \cdot s_{n-2d} - s_{d-1}[s_2] \cdot s_{n-2d+2}
        \}_{\lambda_1 \leq n-2d+a} \cdot q^d
    \end{equation}  
    where we interpret $s_{-1} := 0$.
\end{theorem}

\begin{proof}
    We use downwards induction on $a$. If $a = n$, then $R(\MMM_{n,n}) = \CC$ is the trivial representation in degree 0 which indeed has graded Frobenius image $s_0[s_2] \cdot s_n \cdot q^0 = s_n \cdot q^0$.

    Suppose $a < n$. Lemma~\ref{lem:R-degree-bound} implies that $R(\MMM_{n,a})_d = 0$ for $d > (n-a)/2$. By Lemma~\ref{lem:module-transfer} we have
    \begin{align}
        \grFrob( R(\MMM_{n,a}); q) &= \sum_{d \, = \, 0}^{\frac{n-a}{2}} \Frob(R(\MMM_{n,a})_d) \cdot q^d \\ 
        &= \sum_{d \, = \, 0}^{\frac{n-a}{2}} \{\Frob(R(\MMM_{n,n-2d})_d)\}_{\lambda_1 \leq n - 2d + a} \cdot q^d.
    \end{align}
    Induction on $a$ gives every term of this sum except for the one indexed by $d = (n-a)/2$; we have
    \begin{multline}
    \label{eqn:graded-frobenius-inductive-step}
        \sum_{d \, = \, 0}^{\frac{n-a}{2}} \{\Frob(R(\MMM_{n,n-2d})_d)\}_{\lambda_1 \leq n - 2d + a} \cdot q^d = \\
        \Frob( R(\MMM_{n,a} )_{\frac{n-a}{2}} ) \cdot q^{\frac{n-a}{2}} + 
        \sum_{d \, = \, 0}^{\frac{n-a}{2} - 1} \{ \{  s_d[s_2] \cdot s_{n-2d} - s_{d-1}[s_2] \cdot s_{n-2d+2} \}_{\lambda_1 \leq 2(n-2d)} \}_{\lambda_1 \leq n - 2d + a} \cdot q^d = \\
        \Frob( R(\MMM_{n,a} )_{\frac{n-a}{2}} ) \cdot q^{\frac{n-a}{2}} + 
        \sum_{d \, = \, 0}^{\frac{n-a}{2} - 1} \{  s_d[s_2] \cdot s_{n-2d} - s_{d-1}[s_2] \cdot s_{n-2d+2}  \}_{\lambda_1 \leq n - 2d + a} \cdot q^d
    \end{multline}
    where the second equality used $n - 2d + a \leq 2(n-2d)$ for $d < (n-a)/2$. Orbit harmonics gives the ungraded Frobenius image
    \begin{equation}
        \label{eqn:induction-ungraded-frobenius}
        \grFrob( R(\MMM_{n,a}); 1) = \Frob (R(\MMM_{n,a})) = \Frob( \CC[\MMM_{n,a}]) = s_{(n-a)/2}[s_2] \cdot s_a
    \end{equation}
    where the last equality is Equation~\eqref{eqn:induction-product-identification}. Combining Equations~\eqref{eqn:graded-frobenius-inductive-step} and \eqref{eqn:induction-ungraded-frobenius} gives
    \begin{equation}
        \label{eqn:top-determination}
        \Frob( R(\MMM_{n,a} )_{\frac{n-a}{2}} )  = s_{(n-a)/2}[s_2] \cdot s_a - \sum_{d \, = \, 0}^{\frac{n-a}{2} - 1} \{  s_d[s_2] \cdot s_{n-2d} - s_{d-1}[s_2] \cdot s_{n-2d+2}  \}_{\lambda_1 \leq n - 2d + a},
    \end{equation}
    from which Lemma~\ref{lem:s-identity-one} forces
    \begin{equation}
    \label{eqn:forced-consequence}
        \Frob( R(\MMM_{n,a} )_{\frac{n-a}{2}} )  = 
        \{  s_{(n-a)/2}[s_2] \cdot s_{a} - s_{(n-a-2)/2}[s_2] \cdot s_{a+2}  \}_{\lambda_1 \leq 2a}.
    \end{equation}
    Combining Equations~\eqref{eqn:graded-frobenius-inductive-step} and \eqref{eqn:forced-consequence} completes the proof.
\end{proof}

The authors do not know a combinatorial interpretation of the Hilbert series of $R(\MMM_{n,a})$. It would be interesting to find a statistic $\mathrm{stat}: \MMM_{n,a} \to \ZZ_{\geq 0}$ such that 
\begin{equation}\Hilb(R(\MMM_{n,a});q) = \sum_{w \, \in \, \MMM_{n,a}} q^{\mathrm{stat}(w)}.
\end{equation}

\section{Conclusion}
\label{sec:Conclusion}

If $\Zpoints \subseteq \Mat_{n \times n}(\CC)$ is any set of permutation matrices closed under the conjugation action of $\symm_n$, the orbit harmonics quotient $R(\Zpoints)$ is a $\symm_n$-module. The minimal such loci are the conjugacy classes \begin{equation}
    \KKK_\lambda := \{ w \in \symm_n \,:\, \text{$w$ has cycle type $\lambda$} \}.
\end{equation}
We regard $\KKK_\lambda$ as a set of permutation matrices in $\Mat_{n \times n}(\CC)$.

\begin{problem}
    \label{prob:other-conjugacy-classes}
    For arbitrary $\lambda \vdash n$, calculate the graded $\symm_n$-structure of $R(\KKK_\lambda)$.
\end{problem}

Section~\ref{sec:Conjugacy} solves Problem~\ref{prob:other-conjugacy-classes} when the parts of $\lambda$ have size $\leq 2$.
 Problem~\ref{prob:other-conjugacy-classes} is likely to be extremely difficult in general. Indeed, for $p > 0$ define a symmetric function $L_p \in \Lambda_p$ by 
\begin{equation}
    L_p := \sum_{p \, \mid \, \maj(T)} s_{\lambda(T)}
\end{equation}
where 
\begin{itemize}
    \item the {\em major index} $\maj(T)$ of a standard tableau $T$ with $p$ boxes is the sum over all $1 \leq i \leq p-1$ such that $i$ appears in a higher row of $T$ than $i+1$, 
    \item the sum is over all standard tableaux $T$ with $p$ boxes with major index divisible by $p$, and
    \item $\lambda(T) \vdash p$ is the shape of the tableau $T$.
\end{itemize}
The symmetric function $L_p$ is the Frobenius image of the coset module $\CC[\symm_p/C_p]$ where $C_p = \langle (1,\dots,p)\rangle \subseteq \symm_p$. It follows that if $\lambda$ has part multiplicities  $\lambda = (1^{a_1} 2^{a_2} 3^{a_3} \cdots )$ we have
\begin{equation}
    \Frob(\CC[\KKK_\lambda]) = \prod_{p \, \geq \, 1} s_{a_p}[L_p].
\end{equation}
The Schur expansion of $s_{a_p}[L_p]$ is not known in general, so the Schur $\grFrob(R(\KKK_\lambda);q)$ is likely beyond current technology. There is a combinatorial monomial expansion of $s_{a_p}[L_p]$ defined in terms of necklaces. It may be interesting to see if there is a nice monomial expansion of $\grFrob(R(\KKK_\lambda);q)$ involving a $q$-statistic on necklaces.

Another open problem concerns log-concavity. Recall that a sequence $(a_1, \dots, a_m)$ of positive real numbers is {\em log-concave} if for all $1 < i < m$ we have $a_i^2 \geq a_{i-1} a_{i+1}$. Chen conjectured \cite{Chen} that the sequenece $(a_{n,1}, \dots, a_{n,n})$ is log-concave, where $a_{n,k}$ counts permutations in $\symm_n$ whose longest increasing subsequence has length $k$. B\'ona, Lackner, and Sagan conjectured \cite{BLS} that $(b_{n,1}, \dots, b_{n,n})$ is log-concave where $b_{n,k}$ counts involutions in $\symm_n$ with longest decreasing subsequence of length $k$. Both of these conjectures are open as of this writing. In terms of orbit harmonics, Chen's conjecture states that the Hilbert series of the permutation matrix orbit harmonics ring $R(\symm_n)$ is log-concave.

One way to decorate the notion of log-concavity is to incorporate equivariance. If $G$ is a group, a sequence $(V_1, \dots, V_m)$ of $G$-modules is {\em  $G$-log-concave} if for all $1 < i < m$ there exists a $G$-equivariant embedding $V_{i-1} \otimes V_{i+1} \hookrightarrow V_i \otimes V_i$. Here $G$ acts diagonally on the tensor products $V_{i-1} \otimes V_{i+1}$ and $V_i \otimes V_i$, i.e. $g \cdot (v \otimes v') = (g \cdot v) \otimes (g \cdot v')$. Finally, a graded $G$-module $V = \bigoplus_{d = 0}^m V_d$ is called  $G$-log-concave if the sequence $(V_0, V_1, \dots, V_m)$ of graded pieces has this property. Equivariant log-concavity has attracted increasing attention in recent years \cite{GX, KR, Li, NR}.

The orbit harmonics ring $R(\symm_n)$ for the full set of permutation matrices carries an action of the product group $\symm_n \times \symm_n$; Rhoades conjectured \cite{RhoadesViennot} that $R(\symm_n)$ is  $(\symm_n \times \symm_n)$-log-concave. We conjecture that our modules are also equivariantly log-concave.

\begin{conjecture}
    \label{conj:equivariantly-log-concave}
    The graded $\symm_n$-modules $R(\MMM_n)$ and $R(\MMM_{n,a})$ are $\symm_n$-log-concave.
\end{conjecture}

Conjecture~\ref{conj:equivariantly-log-concave} has been checked for $n \leq 15$. In particular, Conjecture~\ref{conj:equivariantly-log-concave} would imply that the graded rings in question have Hilbert series with log-concave coefficient sequences. It is not hard to check this non-equivariant log-concavity in the case of $R(\MMM_n)$ using Corollary~\ref{cor:matching-hilb}.

Another possible direction for future studies is to extend the results to colored permutation groups. Let $n,r$ be positive integers, the {\em $r$-colored permutation group} $\symm_{n,r}$ is the wreath product $\ZZ_r \wr \symm_n$, where $\ZZ_r$ is the cyclic group of size $r$. In $\Mat_{n \times n}(\CC)$, we have
\begin{equation}
    \symm_{n,r} = \left\{ A \in \Mat_{n \times n}(\CC) \,:\, \begin{array}{c} \text{$A$ has exactly one nonzero entry in each row and column,} \\ \text{and nonzero entries of $A$ are $r^{th}$ roots-of-unity}  \end{array} \right\}.
\end{equation}
In particular, the group $\symm_{n,2}$ is the type $B$ Coxeter group of signed permutation matrices. We define 
\begin{equation}\MMM_n^r := \{ w \in \symm_{n,r} \,:\, w^2=1 \} \subseteq \Mat_{n \times n}(\CC).
\end{equation} Note that $\MMM_n^r$ is the collection of Hermitian matrices in $\symm_{n,r}$. 
\begin{problem}\label{prob: colored matching}
    Find the graded $\symm_{n,r}$-structure of $R(\MMM_n^r)$.
\end{problem}
More generally, conjugacy classes in $\symm_{n,r}$ are labelled by {\em $r$-partitions} $\bl = (\lambda^0, \lambda^1, \dots , \lambda^{r-1})$, where each $\lambda^i$ is a partition and $\sum_{i=0}^{r-1} |\lambda^i| = n$. Let
\begin{equation}
    \KKK_{\bl} := \{w \in \symm_{n,r} : w \text{ has cycle type } \bl\}
\end{equation}
where the cycle type of a colored permutation is explicitly
described in Macdonald's book \cite[Cpt. I, Appendix B]{Macdonald}. Regarding $\KKK_{\bl}$ as matrices in $\Mat_{n \times n}(\CC)$, we extend Problem~\ref{prob: colored matching} as follows.
\begin{problem}\label{prob: colored cycle type}
    For an $r$-partition $\bl$ of $n$, calculate the $\symm_{n,r}$ structure of $R(\KKK_{\bl})$.
\end{problem}
The locus $\MMM_n^r$ in Problem~\ref{prob: colored matching} is the disjoint union of the $\KKK_{\bl}$'s where the parts of $\lambda^0$ have size $\leq 2$ and the parts of all other $\lambda^i$'s have size $\leq 1$. Similar to Problem~\ref{prob:other-conjugacy-classes}, it is expected that Problem~\ref{prob: colored cycle type} will be extremely difficult in general. Problem~\ref{prob: colored cycle type} may be more approachable for conjugacy classes of involutions.

\section{Acknowledgements}

B. Rhoades was partially supported by NSF Grant DMS-2246846.

\end{document}